\documentclass[11pt]{amsart}
\usepackage{latexsym}
\usepackage{amsfonts}

\usepackage{amsmath,amsthm,amssymb}

\title[Ping-pong and Outer space]{Ping-pong and Outer space}

\author[I.~Kapovich]{Ilya Kapovich}

\address{\tt Department of Mathematics, University of Illinois at
  Urbana-Champaign, 1409 West Green Street, Urbana, IL 61801, USA
  \newline http://www.math.uiuc.edu/\~{}kapovich/} \email{\tt
  kapovich@math.uiuc.edu}

\author[M.~Lustig]{Martin Lustig}\address{\tt Math\'ematiques
  (LATP), Universit\'e Paul C\'ezanne -Aix Marseille III,\\  av. Escadrille
  Normandie-Ni\'emen, 13397 Marseille 20, France} \email{\tt Martin.Lustig@univ-cezanne.fr}

\newtheorem{thm}{Theorem}[section] \newtheorem{lem}[thm]{Lemma}
\newtheorem{cor}[thm]{Corollary} 
\newtheorem{prop}[thm]{Proposition} \theoremstyle{definition}
\newtheorem{defn}[thm]{Definition}

\newtheorem{conv}[thm]{Convention} \newtheorem{rem}[thm]{Remark}

\def\epsilon{\varepsilon}
\def\phi{\varphi}
\def\Pr{\mathbb P}

\newcommand{\Supp}{\mbox{Supp}}

\newcommand{\Curr}{\mbox{Curr}}
\newcommand{\Out}{\mbox{Out}}
\newcommand{\Aut}{\mbox{Aut}}
\newcommand{\inv}{^{-1}}

\newcommand{\Stab}{\mbox{Stab}}
\newcommand{\Mod}{\mbox{Mod}}

\newcommand{\FN}{F_N}   
\newcommand{\cvn}{\mbox{cv}_N}
\newcommand{\cvnbar}{\overline{\mbox{cv}}_N}
\newcommand{\CVN}{\mbox{CV}_N}
\newcommand{\CVNbar}{\overline{\mbox{CV}}_N}

\newcommand{\PCurr}{\Pr\Curr(\FN)}

\newcommand{\R}{\mathbb R}
\newcommand{\Z}{\mathbb Z}

\def\strutdepth{\dp\strutbox}
\def \ss{\strut\vadjust{\kern-\strutdepth \sss}}
\def \sss{\vtop to \strutdepth{
\baselineskip\strutdepth\vss\llap{$\diamondsuit\;\;$}\null}}

\def\strutdepth{\dp\strutbox}
\def \sst{\strut\vadjust{\kern-\strutdepth \ssss}}
\def \ssss{\vtop to \strutdepth{
\baselineskip\strutdepth\vss\llap{$\spadesuit\;\;$}\null}}

\def\strutdepth{\dp\strutbox}
\def \ssh{\strut\vadjust{\kern-\strutdepth \sssh}}
\def \sssh{\vtop to \strutdepth{
\baselineskip\strutdepth\vss\llap{$\heartsuit\;\;$}\null}}

\def\qed{\hfill\rlap{$\sqcup$}$\sqcap$\par}

\def\strutdepth{\dp\strutbox}
\def \ss{\strut\vadjust{\kern-\strutdepth \sss}}
\def \sss{\vtop to \strutdepth{
\baselineskip\strutdepth\vss\llap{$\diamondsuit\;\;$}\null}}

\def\strutdepth{\dp\strutbox}
\def \sst{\strut\vadjust{\kern-\strutdepth \ssss}}
\def \ssss{\vtop to \strutdepth{
\baselineskip\strutdepth\vss\llap{$\spadesuit\;\;$}\null}}

\def\qed{\hfill\rlap{$\sqcup$}$\sqcap$\par}

\begin{document}

\begin{abstract}
We prove that, if $\phi,\psi\in \Out(\FN)$ are hyperbolic iwips
(irreducible with irreducible powers) such that $\langle
\phi,\psi\rangle\subseteq \Out(\FN)$ is not virtually cyclic, then some high
powers of $\phi$ and $\psi$ generate a free subgroup of rank two for which all
non-trivial elements are again hyperbolic iwips. Being a hyperbolic iwip element of $\Out(\FN)$ is strongly analogous to being a pseudo-Anosov element of a mapping class group, so the above result provides
analogues of ``purely pseudo-Anosov" free subgroups in
$\Out(\FN)$.
\end{abstract}

\thanks{The first author was supported by the NSF
  grants DMS-0603921 and DMS-0904200}

\subjclass[2000]{Primary 20F, Secondary 57M, 37B, 37D}

\maketitle

\tableofcontents

\section{Introduction}

One of the most important sources for understanding free group
automorphisms is the analogy with surface groups and mapping classes.
Many key concepts from Thurston's approach to Teichm\"uller theory
have been successfully carried over to the $\Out(\FN)$ world, most
notably Culler-Vogtmann's Outer space \cite{CV}, and Bestvina-Handel's
train track representatives \cite{BH92}.  However, often the situation
for $\Out(\FN)$ turns out to be more difficult (but also much richer
in interesting phenomena) than in the parallel mapping class cosmos.
One such fundamental situation arises with the translation of the
concept of {\em pseudo-Anosov} homeomorphisms to free group
automorphisms.  It turns out that there are two such possible
translations, both natural and interesting:

The first translation is based on the characterization of
pseudo-Anosov homeomorphisms $h: \Sigma \to \Sigma$ (where $\Sigma$ is
a closed hyperbolic surface) as precisely those which give a mapping
torus $\Sigma \rtimes_h \mathbb S^1$ that admits a hyperbolic structure.

This is equivalent to the condition that $\pi_1(\Sigma \rtimes_h \mathbb S^1)$ is Gromov-hyperbolic.  By analogy, one can consider {\em hyperbolic} automorphisms $\Phi: \FN \to \FN$, defined by the fact that the mapping torus group $G_\Phi = \FN \rtimes_\Phi \Z$ is word-hyperbolic. The Bestvina-Feighn Combination Theorem~\cite{BF92} implies that an automorphism $\Phi$ of $\FN$ is hyperbolic if and only if for some (and hence any) free basis $A$ of $\FN$ there exist $M\ge 1$ and $\lambda>1$ such that for every $w\in \FN$ we have $\lambda|w|_A\le \max\{|\Phi^M(w)|_A, |\Phi^{-M}(w)|_A\}$. This latter condition is often taken as the definition of an automorphism $\Phi$ of $\FN$ for being hyperbolic. It is not hard to see that
whether $\Phi\in \Aut(\FN)$ is hyperbolic or not
depends only on the outer automorphism class $\phi$ of $\Phi$ in $\Out(\FN)$. An important result of Brinkmann~\cite{Br} shows that $\phi\in \Out(\FN)$ is hyperbolic if and only if $\phi$ is \emph{atoroidal}, that is, if there does not exist a non-trivial conjugacy class in $\FN$ that is fixed by some positive power of $\phi$.

The second
translation of the notion of being pseudo-Anosov to the free
group setting is based on the dynamical properties of pseudo-Anosov
homeomorphisms: a homeomorphism $h: \Sigma \to \Sigma$ is pseudo-Anosov
if and only if $h$ and hence any positive power of $h$ is not
reducible. However, in the free group setting the notion of {\em
  reducible} automorphisms $\Phi: \FN \to \FN$ is much more delicate
than for surfaces: If $h$ fixes (up to isotopy) an essential
subsurface of $\Sigma$, than it also fixes the complementary
subsurface. But is is easy to find examples where $\Phi$ fixes a
proper free factor of $\FN$,
but no complementary
free factor is is mapped to a conjugate of itself.

In this context, the notion of being \emph{irreducible} for elements of $\Out(\FN)$ (see Definition~\ref{defn:irr} below) has been proposed in \cite{BH92}, but contrary to ``pseudo-Anosov'',
the property ``irreducible'' is not
stable under replacing the automorphism by
a positive power.
More useful seems the following notion:
An element $\phi\in \Out(F_N)$
(or any of its lifts $\Phi \in \Aut(F_N)$) is said to be
\emph{irreducible with irreducible powers} or an \emph{iwip} for
short, if for every $n\ge 1$
the power
$\phi^n$ is irreducible (sometimes such
automorphisms are also called \emph{fully irreducible}). It is not hard to see that $\phi\in
\Out(\FN)$ is an iwip if and only if no positive power of $\phi$
preserves the conjugacy class of a proper free factor of $\FN$ (and one can take the latter condition as the definition of being an iwip). It is
easy to construct examples of elements of $\Out(F_N)$ that are
hyperbolic but reducible. Similarly, there exists non-hyperbolic iwips
(they come from pseudo-Anosov homeomorphisms of once-punctured
surfaces). Thus the notions of being iwip and being hyperbolic are
logically independent.

Both of these free group analogues of being pseudo-Anosov play an important role in the study of $\Out(\FN)$.  Iwips have nicer properties: for example, they act with ``North-South" dynamics on the Thurston compactification of Outer space~\cite{LL} (just as pseudo-Anosovs do on Teichm\"uller space and its Thurston boundary).  Hyperbolic automorphisms, on the other hand, are easier to come by. For example, it has been shown by Bestvina, Feighn and Handel~\cite{BFH97} that every subgroup of $\Out(\FN)$, which contains a hyperbolic iwip and which is not virtually cyclic, contains a free subgroup of rank two where every non-trivial element is
hyperbolic.

The main result of this paper is the analogous statement of this
last result for hyperbolic iwips (c.f. Theorem~\ref{iwips-only1} below):

\begin{thm}\label{thm:A}
Let $N\ge 3$ and let
$\phi,\psi\in \Out(F_N)$ be hyperbolic iwips such that the subgroup $\langle \phi,\psi\rangle \subseteq \Out(F_N)$ is not virtually cyclic.
Then there exists $m,n\ge 1$ such that the subgroup $G=\langle \phi^m,\psi^n\rangle \subseteq \Out(F_N)$ is free of rank two and such that every nontrivial element of $G$ is again a hyperbolic iwip.
\end{thm}

Thus the group $G$ in Theorem~\ref{thm:A} is ``purely hyperbolic iwip''.
It was already known by the results of \cite{BFH97} that one can ensure for every nontrivial element of $G$ as in Theorem~\ref{thm:A} to be a hyperbolic automorphism, and the new result here is the iwip property. Nevertheless, we also provide a complete and independent proof of the "purely hyperbolic'' property as well. 

By Corollary~\ref{cor:pair} for two hyperbolic iwips $\phi,\psi\in
\Out(F_N)$ the condition that they don't generate a virtually cyclic
subgroup is equivalent to the condition that they don't have any
common non-trivial powers, that is $\langle \phi\rangle\cap \langle
\psi\rangle=\{1\}$.

As a consequence of Theorem~\ref{thm:A}, we obtain
(c.f. Corollary~\ref{cor:sgp} below):

\begin{cor}
Let $G\subseteq \Out(F_N)$ be a non-virtually-cyclic
subgroup that contains a hyperbolic iwip.
Then $G$ contains a non-abelian
free subgroup where all
non-trivial elements are hyperbolic iwips.
\end{cor}

Note that results similar to the statement of Theorem~\ref{thm:A} play an important role in the study of mapping class groups. Namely, for the mapping class group $\Mod(\Sigma)$ of a closed hyperbolic
surface $\Sigma$ it is interesting to find \emph{purely pseudo-Anosov}
subgroups of $\Mod(\Sigma)$,  i.e. subgroups where all non-trivial elements are pseudo-Anosov. One of the motivations in looking for purely pseudo-Anosov subgroups of $\Mod(\Sigma)$ is in trying to find new examples of word-hyperbolic extensions of $\pi_1(\Sigma)$ by groups other than infinite cyclic ones. An important early example of a non-abelian free purely pseudo-Anosov subgroup
$\Mod(\Sigma)$ is due to Mosher~\cite{Mosher} who used it to construct a word-hyperbolic extension of $\pi_1(\Sigma)$ by the free group $F_2$. Mosher's example was based on exploiting ping-pong considerations for the action of Schottky-type subgroups of $\Mod(\Sigma)$ on the boundary of the Teichmuller space; these types of subgroups are basic examples of convex-cocompact subgroups of mapping class groups. Another important source of purely pseudo-Anosov subgroups of mapping class groups comes from the work of Whittlesey~\cite{W}. This topic plays a key role in the theory of \emph{convex-cocompact} subgroups of mapping class groups~\cite{FM,Ha05,KeLe07,KeLe08}, and it is known that every such convex-cocompact subgroup is purely pseudo-Anosov.

A recent result of Handel and Mosher~\cite{HM} characterizes those subgroups of $\Out(F_N)$ that do not contain an iwip and shows that such subgroups have a rather special structure: if a subgroup $G\subseteq \Out(F_N)$ does not contain an iwip then there is a subgroup of finite index $H\subseteq G$ such that $H$ preserves the conjugacy class of a proper free factor of $F_N$.

Theorem~\ref{thm:A} can also be derived from a recent result of Bestvina and Feighn~\cite{BF08} about the existence of a hyperbolic graph with an $\Out(F_N)$-action, given a finite collection of independent hyperbolic iwips. Our proof is based on rather different and more direct arguments and we believe that it has substantial independent value, especially in view of the goal of developing the theory of convex-cocompactness for subgroups of $Out(F_N)$. Note also, that a new paper of Clay and Pettet~\cite{CP} provides a proof of a related statement to our Theorem~\ref{thm:A}: they prove that given two "sufficiently transverse" Dehn twists $\phi,\psi\in \Out(F_N)$, for some sufficiently large $m,n\ge 1$ the subgroup $\langle \phi^m, \psi^n\rangle\le \Out(F_N)$ is free of rank two and every nontrivial element of that subgroup, except those that are conjugate to powers of $\phi^m$, $\psi^n$, is a hyperbolic iwip. This result of Clay and Pettet and our Theorem~\ref{thm:A} are logically independent and the proofs are very different.
Theorem~\ref{thm:A} was applied and pushed further in a new paper of Hamenst\"adt~\cite{Ha09}. 

\smallskip

We establish Theorem~\ref{thm:A} via studying the dynamics of the action of $\Out(F_N)$ and of its subgroups on the space $\cvnbar$ of very small isometric $\mathbb R$-tree actions of $F_N$ (which is the closure of the Outer space $\cvn$ in the length function topology) and on the space $\Curr(F_N)$ of \emph{geodesic currents} on $F_N$. A geodesic current is a measure-theoretic
analogue
of the notion of a conjugacy class in a free group (or a free homotopy class of a closed curve on the surface). Geodesic currents in the context of hyperbolic surfaces were introduced by Bonahon who used them to study the geometry of the Teichm\"uller space~\cite{Bo86,Bo88}. In the context of free groups geodesic currents were first introduced in
the Ph.D.-thesis
of Reiner Martin~\cite{Martin} and later re-introduced and studied systematically by Kapovich~\cite{Ka1,Ka2,Ka3}, Kapovich-Lustig~\cite{KL1,KL2,KL3} and others~\cite{Fra,KN}. Recent applications of geodesic currents include results related to free group
analogues
of the curve complex (Kapovich-Lustig~\cite{KL2}, Bestvina-Feighn~\cite{BF08}) and to bounded cohomology of $\Out(F_N)$ and of its subgroups (Hamenstadt~\cite{Ha}). A key component in these results, as well as in the proofs of the main results of the present
paper, is the \emph{geometric intersection form}.
The latter
pairs very small $F_N$-trees and geodesic currents and
shares important features in common with Bonahon's notion of a geometric intersection number between two geodesic currents. This intersection form was initially constructed in \cite{Ka2}, \cite{Lu1}
for the ordinary unprojectivized Outer space $\cvn$ and recently extended
in our joint paper~\cite{KL2} to the closure $\cvnbar$ of $\cvn$.

We start by exploiting the fact that a hyperbolic iwip $\phi\in
\Out(F_N)$ acts with a ``North-South" dynamics on both the
projectivization $\CVNbar$ of $\cvnbar$ and on the projectivization
$\PCurr$ of $\Curr(F_N)$. In the process we introduce,
using the intersection form, natural ``height functions" associated to
$\phi$ on each of $\CVNbar$ and $\PCurr$, which provide
useful stratifications of these spaces. As a corollary of the
``North-South'' dynamics for the action of hyperbolic iwips on $\PCurr$, we obtain a new
proof (Theorem~\ref{thm:hyperb}) of the following result of Bestvina, Feighn and Handel~\cite{BFH97}:
If $\phi,\psi\in \Out(F_N)$ are as in Theorem~\ref{thm:A} and
$\Phi,\Psi\in \Aut(F_N)$ are their representatives in $Aut(F_N)$, then for
sufficiently high powers $\Phi^n,\Psi^m$ of $\Phi$ and $\Psi$, the semi-direct product
$G_{n,m}=F_N\rtimes \langle \Phi^n,\Psi^m\rangle$ is word-hyperbolic.
In this case the
subgroups $\langle \phi^n,\psi^m\rangle\subseteq \Out(F_N)$ and $\langle
\Phi^n,\Psi^m\rangle\subseteq \Aut(F_N)$ are free of rank two
and, as noted in Remark~\ref{purely-hyperbolic} below, the hyperbolicity
of $G_{n,m}$ already implies that every non-trivial element of $\langle
\phi^n,\psi^m\rangle$ is hyperbolic.

Establishing the ``purely iwip'' part of Theorem~\ref{thm:A} requires a much more delicate analysis and new tools and ideas, in order to rule out the existence of non-trivial reducible elements in free subgroups of $\Out(F_N)$ generated by two large powers of hyperbolic iwips. In particular, we exploit
the interplay between the
right ping-pong action of such subgroups on $\cvnbar$ and their
simultaneous left ping-pong action on $\PCurr$.
Thus ping-pong arguments play a key role in the proof
Theorem~\ref{thm:A}. Note that ping-pong type arguments, in different
settings, are also important in the proof of the Tits Alternative for
$\Out(\FN)$ by Bestvina, Feighn and Handel~\cite{BFH97,BFH00,BFH05}. Also,
ping-pong arguments for iwips and the existence of Schottky-type free
subgroups in $\Out(F_N)$ yielded by such arguments, are a key tool in
the proof by Bridson and de la Harpe~\cite{BH} that $\Out(F_N)$ is
$C^\ast$-simple for $N\ge 3$. Both \cite{BFH97} and \cite{BH} use
the ping-pong arguments (and their consequences) for the action of
$\Out(F_N)$ on the set of the ``legal'' or ``stable'' laminations
associated to all iwip elements of $\Out(F_N)$ and exploit the fact
that this set admits an $\Out(F_N)$-equivariant embedding in
$\CVNbar$.

A careful analysis of the proof of Theorem~\ref{thm:A}
shows that its
conclusion holds for an arbitrary finite number
of hyperbolic iwips:

\begin{cor}
\label{more-than-two}
 Let $\phi_1,\dots, \phi_k\in \Out(F_N)$
 be hyperbolic iwips such that
 for every $1\le i<j\le k$ the subgroup $\langle
\phi_i,\phi_j\rangle$ is not virtually cyclic.
Then there exist $n_1, \dots, n_k \ge 1$ such that the subgroup $G=\langle \phi_1^{n_1}, \dots,\phi_k^{n_k}\rangle \subseteq \Out(F_N)$ is free of rank $k$, and such that every nontrivial element of $G$ is again a hyperbolic iwip.
\end{cor}

An interesting goal for future work would be to develop a theory of
``convex-cocompact" subgroups of $\Out(F_N)$ that resembles the theory
of convex-cocompact subgroups of mapping class groups.
A first step for such a theory is given in
\cite{KL5}.
We informally
call the free subgroups of $\Out(\FN)$ generated by two large powers
of hyperbolic iwips, that appear in the conclusion of
Theorem~\ref{thm:A}, \emph{Shottky-type subgroups} of $\Out(F_N)$. We
believe that Schottky-type subgroups should provide basic examples of
convex-cocompact subgroups of of $\Out(F_N)$ and we hope that
analyzing their properties will lead to a successful formulation of
the convex-cocompactness theory in the $\Out(F_N)$ context.

\smallskip
\noindent
{\em Acknowledgement:}

The authors are grateful to Chris Leininger for useful conversations.
The authors also thank the referee for a careful reading of the paper and for many helpful suggestions.

\section{Outer space and the space of geodesic currents}

We give here only a brief overview of basic facts related to Outer space and the space of geodesic currents. We refer the reader to~\cite{CV,Ka2} for more detailed background information.

\subsection{Outer space} Let $N\ge 2$. The \emph{unprojectivized Outer space} $\cvn$ consists of all minimal free and discrete isometric actions on $F_N$ on $\mathbb R$-trees (where two such actions are considered equal if there exists an $F_N$-equivariant isometry between the corresponding trees). There are several different topologies on $\cvn$ that are known to coincide, in particular the equivariant Gromov-Hausdorff convergence topology and the so-called \emph{axis} or \emph{length function} topology. Every $T\in \cvn$ is uniquely determined by its \emph{translation length function} $||.||_T:F_N\to\mathbb R$, where $||g||_T$ is the translation length of $g$ on $T$. Two trees $T_1,T_2\in\cvn$ are close if the functions $||.||_{T_1}$ and $||.||_{T_1}$ are close pointwise on a large ball in $F_N$. The closure $\cvnbar$ of $\cvn$ in either of these two topologies is well-understood and known to consists precisely of all the so-called \emph{very small} minimal isometric actions of $F_N$ on $\mathbb R$-trees, see \cite{BF93} and \cite{CL}.
The outer automorphism group $\Out(F_N)$ has a natural continuous right action on $\cvnbar$ (that leaves $\cvn$ invariant) given at the level of length functions as follows: for $T\in \cvnbar$ and $\phi\in \Out(F_N)$ we have $||g||_{T\phi}=||\Phi(g)||_T$,
with
$g\in F_N$
and
$\Phi\in \Aut(F_N)$ representing $\phi \in \Out(F_n)$.
In terms of tree actions, $T\phi$ is equal to $T$ as a metric space, but the action of $F_N$ is modified
to give
$g\underset{T\phi}{\cdot} x=\Phi(g)\underset{T}{\cdot} x$,
 where $x\in T$, $g\in F_N$ are arbitrary and where
$\Phi\in \Aut(F_N)$
represents as before
of the outer automorphism $\phi$.
The \emph{projectivized Outer space} $\CVN=\mathbb P\cvn$ is defined as the quotient $\cvn/\sim$ where for $T_1\sim T_2$ whenever $T_2=cT_1$ for some $c>0$. One similarly defines the projectivization $\CVNbar=\mathbb P\cvnbar$ of $\cvnbar$ as $\cvnbar/\sim$ where $\sim$ is the same as above. The space $\CVNbar$ is compact and contains $\CVN$ as a dense $\Out(F_N)$-invariant subset. The compactification $\CVNbar$ of $\CVN$ is a free group analogue of the Thurston compactification of the Teichm\"uller space. For $T\in \cvnbar$ its $\sim$-equivalence class is denoted by $[T]$, so that $[T]$ is the image of $T$ in $\CVNbar$. The unprojectivized Outer space $\cvn$ contains an $\Out(F_N)$-invariant closed subspace $cv_N^1$ which is $\Out(F_N)$-equivariantly homeomorphic to $\CVN$. Namely, $cv_N^1$ consists of all trees $T\in \cvn$ such that the quotient metric graph $T/F_N$ has volume one (that is, the sum of the lengths of its edges is equal to one). Many sources identify $\CVN$ and $cv_N^1$ but we will distinguish these objects in the present paper.

\subsection{Geodesic currents} Let $\partial^2F_N:=\{ (x,y)| x,y\in \partial F_N, x\ne y\}$. The action of $F_N$ by translations on its hyperbolic boundary $\partial F_N$ defines a natural diagonal action of $F_N$ on $\partial^2 F_N$. A \emph{geodesic current} on $F_N$ is a positive Radon measure on $\partial^2 F_N$ that is $F_N$-invariant and is also invariant under the ``flip" map $\partial^2 F_N\to \partial^2 F_N$, $(x,y)\mapsto (y,x)$. The space $\Curr(F_N)$ of all geodesic currents on $F_N$ has a natural $\mathbb R_{\ge 0}$-linear structure and is equipped with the
weak*-topology
of pointwise convergence on continuous functions. Every point $T\in \cvn$ defines a \emph{simplicial chart} on $\Curr(F_N)$ which allows one to think about geodesic currents as systems of nonnegative weights satisfying certain Kirchhoff-type equations; see \cite{Ka2} for details. We briefly recall the simplicial chart construction for the case where $T_A\in \cvn$ is the Cayley tree corresponding to a free basis $A$ of $F_N$. For a nondegenerate geodesic segment $\gamma=[p,q]$ in $T_A$ the \emph{two-sided cylinder} $Cyl_A(\gamma)\subseteq \partial^2 F_N$ consists of all $(x,y)\in \partial^2 F_N$ such that the geodesic from $x$ to $y$ in $T_A$ passes through $\gamma=[p,q]$. Given a nontrivial freely reduced  word $v\in F(A)=F_N$ and a current $\mu\in \Curr(F_N)$, the ``weight" $\langle v,\mu\rangle_A$ is defined as $\mu(Cyl_A(\gamma))$ where $\gamma$ is any segment in the Cayley graph $T_A$ labelled by $v$ (the fact that the measure $\mu$ is $F_N$-invariant implies that a particular choice of $\gamma$ does not matter). A current $\mu$ is uniquely determined by a family of weights $\big(\langle v,\mu\rangle_A\big)_{v\in F_N-\{1\}}$. The
weak*-topology
on $\Curr(F_N)$ corresponds to pointwise convergence of the weights for every $v\in F_N, v\ne 1$.

There is a natural left action of $\Out(F_N)$ on $\Curr(F_N)$ by continuous linear transformations. Specifically, let $\mu\in \Curr(F_N)$, $\phi\in \Out(F_N)$ and let $\Phi\in \Aut(F_N)$ be a representative of $\phi$ in $Aut(F_N)$. Since $\Phi$ is a quasi-isometry of $F_N$, it extends to a homeomorphism of $\partial F_N$ and, diagonally, defines a homeomorphism of $\partial^2 F_N$. The measure $\phi\mu$ on $\partial^2 F_N$ is defined as follows. For a Borel subset $S\subseteq \partial^2 F_N$ we have $(\phi\mu)(S):=\mu(\Phi^{-1}(S))$. One then checks that $\phi\mu$ is a current and that it does not depend on the choice of a representative $\Phi$ of $\phi$.

For every $g\in F_N, g\ne 1$ there is an associated \emph{counting current} $\eta_g\in \Curr(F_N)$. If $A$ is a free basis of $F_N$ and the conjugacy class $[g]$ of $g$ is realized by a ``cyclic word" $W$ (that is a cyclically reduced word in $F(A)$ written on a circle with no specified base-vertex), then for every nontrivial freely reduced word $v\in F(A)=F_N$ the weight $\langle v,\eta_g\rangle_A$ is equal to the total number of occurrences of $v^{\pm 1}$ in $W$ (where an occurrence of $v$ in $W$ is a vertex on $W$ such that we can read $v$ in $W$ clockwise without going off the circle). We refer the reader to \cite{Ka2} for a detailed exposition on the topic. By construction the counting current $\eta_g$ depends only on the conjugacy class $[g]$ of $[g]$ and it also satisfies $\eta_g=\eta_{g^{-1}}$. One can check~\cite{Ka2} that for $\phi\in \Out(F_N)$ and $g\in F_N, g\ne 1$ we have $\phi\eta_g=\eta_{\phi(g)}$. Scalar multiples $c\eta_g\in \Curr(F_N)$, where $c\ge 0$, $g\in F_N, g\ne 1$ are called \emph{rational currents}. A key fact about $\Curr(F_N)$ states that the set of all rational currents is dense in $\Curr(F_N)$.

The \emph{space of projectivized geodesic currents} is defined as $\PCurr=\Curr(F_N)-\{0\}/\sim$ where $\mu_1\sim\mu_2$ whenever there exists $c>0$ such that $\mu_2=c\mu_1$. The $\sim$-equivalence class of $\mu\in \Curr(F_N)-\{0\}$ is denoted by $[\mu]$.  The action of $\Out(F_N)$ on $\Curr(F_N)$ descends to a continuous action of $\Out(F_N)$ on $\PCurr$. The space $\PCurr$ is compact and the set $\{[\eta_g]\: g\in F_N, g\ne 1 \}$ is a dense subset of it.

\subsection{Intersection form}

In \cite{KL2} we constructed a natural geometric \emph{intersection form} which pairs trees and currents:

\begin{prop}\label{int}\cite{KL2}
Let $N\ge 2$. There exists a unique continuous map $\langle\ ,\ \rangle : \cvnbar \times \Curr(F_N)\to \mathbb R_{\ge 0}$ with the following properties:
\begin{enumerate}
\item We have $\langle T, c_1\mu_1+c_2\mu_2\rangle=c_1\langle T,\mu_1\rangle+c_2\langle T,\mu_2\rangle$ for any $T\in \cvnbar$, $\mu_1,\mu_2\in \Curr(F_N)$, $c_1,c_2\ge 0$.
\item We have $\langle cT, \mu\rangle=c\langle T,\mu\rangle$ for any $T\in\cvnbar$, $\mu\in \Curr(F_N)$ and $c\ge 0$.
\item We have $\langle T\phi,\mu\rangle=\langle T, \phi\mu\rangle$ for any $T\in\cvnbar$, $\mu\in \Curr(F_N)$ and $\phi\in \Out(F_N)$.
\item We have $\langle T, \eta_g\rangle=||g||_T$ for any $T\in \cvnbar$ and $g\in F_N, g\ne 1$.
\end{enumerate}
\end{prop}

Note that here we work with the \emph{right} action of $\Out(F_N)$ on $\cvnbar$, which is related to the \emph{left} action  of $\Out(F_N)$ on $\cvnbar$ considered in \cite{KL2} via $T\phi= \phi^{-1}T$, where $T\in \cvnbar$, $\phi\in \Out(F_N)$. This accounts for the difference in how part (3) of Proposition~\ref{int} is stated above compared with the formulation of the main result in \cite{KL2}.

\section{Stabilizers of eigentrees and eigencurrents}\label{sect:stab}

\begin{defn}\label{defn:irr}
An element $\phi\in \Out(F_N)$ is \emph{reducible} if there
exists a free product decomposition $F_N=C_1\ast\dots C_k\ast F'$,
where $k\ge 1$ and $C_i\ne \{1\}$, such that $\phi$ permutes the
conjugacy classes of subgroups $C_1,\dots, C_k$ in $F_N$. An element
$\phi\in \Out(F_N)$ is called \emph{irreducible} if it is not
reducible. An element $\phi\in \Out(F_N)$ is said to be
\emph{irreducible with irreducible powers} or an \emph{iwip} for
short, if for every $n\ge 1$
the power
$\phi^n$ is irreducible (sometimes such
automorphisms are also called \emph{fully irreducible}).
This is equivalent to the property that no positive power of $\phi$ fixes a conjugacy class of a proper free factor of $\FN$

An outer automorphism $\phi\in \Out(F_N)$ is {\em hyperbolic} or {\em atoroidal} if no positive power of $\phi$ fixes the conjugacy class of a nontrivial element of $\FN$. 

An automorphism $\Phi\in \Aut(F_N)$, is called  {\em hyperbolic} or {\em atoroidal}  if the outer automorphism $\phi\in \Out(F_N)$ defined by $\Phi$ is atoroidal.
\end{defn}

A result of Brinkmann~\cite{Br}, together with the Combination Theorem of Bestvina and Feighn~\cite{BF92}, implies that $\Phi\in \Aut(F_N)$ is atoroidal if and only if the mapping torus group $F_N\rtimes_\Phi \mathbb Z$ is word-hyperbolic.

The following result is due to Reiner Martin~\cite{Martin}:
\begin{prop}\label{prop:rm}
Let $\phi\in \Out(F_N)$ be a hyperbolic iwip. Then there exist unique $[\mu_+],[\mu_-]\in \PCurr$ with the following properties:
\begin{enumerate}
\item The elements $[\mu_+],[\mu_-]\in \PCurr$ are the only fixed points of $\phi$ in $\PCurr$.
\item For any $[\mu]\ne [\mu_-]$ we have $\lim_{n\to\infty} \phi^n[\mu]=[\mu_+]$ and for any $[\mu]\ne [\mu_+]$ we have $\lim_{n\to\infty} \phi^{-n}[\mu]=[\mu_-]$.
\item We have $\phi\mu_+=\lambda_+\mu_+$ and $\phi^{-1}\mu_-=\lambda_-\mu_-$ where $\lambda_+>1$ and $\lambda_->1$. Moreover $\lambda_+$ is the Perron-Frobenius eigenvalue of any train-track representative of $\phi$ and $\lambda_-$ is the Perron-Frobenius eigenvalue of any train-track representative of $\phi^{-1}$.
\end{enumerate}
\end{prop}
Note that if $\phi\in \Out(F_N)$ is a non-hyperbolic iwip, the conclusion of Proposition~\ref{prop:rm} still holds if $\PCurr$ is replaced by the \emph{minimal set} $\mathcal M\subseteq \PCurr$, where $\mathcal M$ is the closure in $\PCurr$ of the set of all $[\eta_a]$, where $a\in \FN$ is a primitive element (see \cite{Martin}).

A similar statement is known for $\CVNbar$ by a result of Levitt and Lustig~\cite{LL}:

\begin{prop}\label{prop:LL}
Let $\phi\in \Out(F_N)$ be an iwip. Then there exist unique $[T_+],[T_-]\in \CVNbar$ with the following properties:
\begin{enumerate}
\item The elements $[T_+],[T_-]\in \CVNbar$ are the only fixed points of $\phi$ in $\CVNbar$.
\item For any $[T]\in \CVNbar$, $[T]\ne [T_-]$ we have $\lim_{n\to\infty} [T\phi^n]=[T_+]$ and for any $[T]\in \CVNbar$, $[T]\ne [T_+]$ we have $\lim_{n\to\infty} [T\phi^{-n}]=[T_-]$.
\item We have $T_+\phi=\lambda_+T$ and $T_-\phi^{-1}=\lambda_-T_-$   where $\lambda_+>1$ and $\lambda_->1$. Moreover $\lambda_+$ is the Perron-Frobenius eigenvalue of any train-track representative of $\phi$ and $\lambda_-$ is the Perron-Frobenius eigenvalue of any train-track representative of $\phi^{-1}$.
\end{enumerate}
\end{prop}

Moreover, in both \cite{LL} and \cite{Martin} it is proved that the
convergence to $[T_+]$ and $[\mu_+]$ in the above statements is
uniform on compact subsets. More precisely:

\begin{prop}\label{uniform}
Let $\phi\in \Out(F_N)$ be a hyperbolic iwip. Let $[T_+],[T_-]\in
\CVNbar$ and $[\mu_+], [\mu_-]\in \PCurr$ be as in
Proposition~\ref{prop:LL} and Proposition~\ref{prop:rm} above.

Let $U, U'$ be open neighborhoods in $\CVNbar$ of $[T_+]$ and $[T_-]$ respectively and let $V, V'$ be open neighborhoods
in $\PCurr$ of
$[\mu_+]$ and $[\mu_-]$ respectively.
Then there exists a constant
$M\ge 1$ such that for every $n\ge M$ we have
$(\CVNbar-U')\phi^n\subseteq U$ and $\phi^n(\PCurr-V')\subseteq V$.
\end{prop}

Proposition~\ref{uniform} immediately implies, via the standard
ping-pong argument, the following:

\begin{cor}\label{cor:pingpong}
Let $\phi,\psi\in \Out(F_N)$ be hyperbolic iwips such that
$[T_\pm(\phi)]$, $[T_\pm(\psi)] \in \CVNbar$ are four distinct points or
that $[\mu_\pm(\phi)], [\mu_\pm(\psi)] \in \PCurr$ are
four distinct points. Then there exists $M\ge 1$ such that for every
$m,n\ge M$ the subgroup $\langle \phi^m, \psi^n\rangle\subseteq \Out(F_N)$
is free of rank two with free basis $\phi^m, \psi^n$.
%
\qed
\end{cor}

We show in Section~\ref{NS} below (specifically, see
Proposition~\ref{prop:sum}) that Proposition~\ref{uniform}
can actually be formally derived from pointwise convergence to $[T_+]$
and $[\mu_+]$ in Proposition~\ref{prop:LL} and
Proposition~\ref{prop:rm}.
This will give an alternative
proof of Proposition~\ref{uniform}.

In \cite{KL3} we gave a characterization of the situation where
$\langle T,\mu\rangle=0$, in terms of the \emph{dual algebraic
  lamination} $L^2(T)$ of the $\mathbb R$-tree $T$ and the {\em
  support} $\Supp(\mu)$ of the current $\mu$ (see \cite{KL3} for a
precise definition of these terms):

\begin{thm}\label{thm:zero}\cite{KL3}
Let $T\in \cvnbar$ and $\mu\in \Curr(F_N)$. Then $\langle T,\mu\rangle=0$ if and only if $\Supp(\mu)\subseteq L^2(T)$.
\end{thm}

As a consequence, we proved~\cite{KL3}:

\begin{prop}\label{prop:zero}\cite{KL3}
Let $\phi\in \Out(F_N)$ be a hyperbolic iwip, and let $[T_+],[T_-]\in
\CVNbar$ and $[\mu_+], [\mu_-]\in \PCurr$ be as in
Proposition~\ref{prop:LL} and Proposition~\ref{prop:rm} above.
Then the following hold:
\begin{enumerate}
\item For $T\in \overline{cv}(F_N)$ we have $\langle T,\mu_+\rangle=0$ if and only if $[T]=[T_-]$ and we have $\langle T,\mu_-\rangle=0$ if and only if $[T]=[T_+]$.
\item For $\mu\in \Curr(F_N)$, $\mu\ne 0$ we have $\langle T_+,\mu\rangle=0$ if and only if $[\mu]=[\mu_-]$ and we have $\langle T_-,\mu\rangle=0$ if and only if $[\mu]=[\mu_+]$.
\item
In particular, we have $\langle T_+, \mu_+\rangle>0$ and $\langle T_-,\mu_-\rangle>0$.
\end{enumerate}
\end{prop}

Note also that, as a direct comparison of the definitions shows, if $\phi\in
\Out(F_N)$ is a hyperbolic iwip, then $\Supp(\mu_+(\phi))$ is exactly what
was termed the ``stable lamination'' $\Lambda_\phi^+$ of $\phi$ in \cite{BFH97}.

A result of \cite{BFH97}, which is reproved in \cite{KL6} via
different methods, states:

\begin{prop}\label{prop:vc}\cite{BFH97,KL6}
Let $N\ge 3$ and let $\phi\in \Out(F_N)$ be an iwip. Then
$\Stab_{Out(F_N)}([T_+(\phi)])$ is virtually cyclic and contains
$\langle \phi\rangle$ as a subgroup of finite index.
\end{prop}

As a consequence, we derive:

\begin{prop}\label{prop:stab}
Let $\phi\in \Out(F_N)$ be a hyperbolic iwip. Then
\begin{gather*}
\Stab_{Out(F_N)}([T_+(\phi)])=\Stab_{Out(F_N)}(\Supp(\mu_-(\phi)))
=\Stab_{Out(F_N)}([\mu_-(\phi)]) \, ,
\end{gather*}
and this stabilizer is virtually cyclic.
\end{prop}

\begin{proof}
Proposition~\ref{prop:zero} implies that
$\Stab_{Out(F_N)}([T_+])=\Stab_{Out(F_N)}([\mu_-])$. Indeed,
suppose $\phi\in \Stab_{Out(F_N)}([T_+])$, so that $T_+\phi=c T_+$ for
some $c>0$. Then
\[
\langle T_+, \phi\mu_-\rangle=\langle T_+\phi, \mu_-\rangle=\langle
cT_+,\mu_-\rangle=c\langle T_+,\mu_-\rangle=0
\]
Therefore $\phi[\mu_-]=[\mu_-]$ by part 2 of
Proposition~\ref{prop:zero} and hence $\Stab_{Out(F_N)}([T_+])\subseteq
\Stab_{Out(F_N)}([\mu_-])$. A similar argument using part (1) of
Proposition~\ref{prop:zero} shows that $\Stab_{Out(F_N)}([\mu_-])\subseteq
\Stab_{Out(F_N)}([T_+])$, and hence
$\Stab_{Out(F_N)}([T_+])=\Stab_{Out(F_N)}([\mu_-])$.

We claim that
$\Stab_{Out(F_N)}(\Supp(\mu_-))=\Stab_{Out(F_N)}([\mu_-])$. Indeed, the
inclusion $\Stab_{Out(F_N)}([\mu_-])\subseteq \Stab_{Out(F_N)}(\Supp(\mu_-))$
is obvious.

Suppose now that $\phi\in \Stab_{Out(F_N)}(\Supp(\mu_-))$. Then
$\Supp(\phi\mu_-)=\Supp(\mu_-)$.

From Proposition~\ref{prop:zero} we know that $\langle
T_+,\mu_-\rangle=0$.
Thus we derive from Theorem \ref{thm:zero} that
$\Supp(\phi\mu_-)=\Supp(\mu_-)$ is
contained in the dual lamination
$L^2(T_+)$ of $T_+$,
so that the converse implication of
Theorem~\ref{thm:zero} implies that $\langle
T_+,\phi\mu_-\rangle=0$. Therefore, by part (2) of Proposition~\ref{prop:zero}, we have $\phi[\mu_-]=[\mu_-]$. Thus
$\Stab_{Out(F_N)}(\Supp(\mu_-))\subseteq \Stab_{Out(F_N)}([\mu_-])$ and hence
$\Stab_{Out(F_N)}(\Supp(\mu_-))=\Stab_{Out(F_N)}([\mu_-])$, as claimed.
\end{proof}

In \cite{BFH97} it is first proved that for an iwip $\phi$ the stabilizer
$\Stab_{Out(F_N)}(\Lambda_\phi^+)$ is virtually cyclic and then that
$\Stab_{Out(F_N)}(\Lambda_\phi^+)=\Stab_{Out(F_N)}([T_+(\phi)])$.
Proposition~\ref{prop:stab} above recovers these results, for a hyperbolic iwip $\phi$, as a consequence of Proposition~\ref{prop:vc}
about virtual cyclicity of $\Stab_{Out(F_N)}([T_+(\phi)])$, for which \cite{KL6} provided an alternative proof to the argument given in \cite{BFH97}.

\begin{prop}\label{prop:dl}
Let $G\subseteq \Out(F_N)$ be a subgroup and such that there exists a hyperbolic iwip $\phi\in G$. Let $[T_+(\phi)],[T_-(\phi)]\in \CVNbar$, $[\mu_+(\phi)],[\mu_-(\phi)]\in \PCurr$ be the attracting and repelling fixed points of $g$ in $\CVNbar$ and $\PCurr$ accordingly.
Then exactly one of the following occurs:
\begin{enumerate}
\item The group $G$ is virtually cyclic and preserves the sets $\{[T_+(\phi)]$, $[T_-(\phi)]\} \subseteq \CVNbar$, $\{[\mu_+(\phi)],[\mu_-(\phi)]\}\subseteq \PCurr$.
\item The group $G$ contains a hyperbolic iwip $\psi=g\phi g^{-1}$ for
  some $g\in G$  such that
  $\{[T_+(\phi)],[T_-(\phi)]\}\cap\{[T_+(\psi)],[T_-(\psi)]\}=\emptyset$ and $\{[\mu_+(\phi)], [\mu_-(\phi)]\}\cap \{[\mu_+(\psi)],[\mu_-(\psi)]\}=\emptyset$. Moreover, in this case there exists $M\ge 1$ such that the subgroup $\langle \phi^M, \psi^M\rangle\subseteq G$ is free of rank two.
\end{enumerate}
\end{prop}

\begin{proof}
Recall that by Proposition~\ref{prop:stab} we have
$\Stab_{Out(F_N)}[T_+(\phi)]=$
\\
$\Stab_{Out(F_N)}[\mu_-(\phi)]$ and
$\Stab_{Out(F_N)}[T_-(\phi)]=\Stab_{Out(F_N)}[\mu_+(\phi)]$ and both of
these are virtually cyclic and contain $\langle\phi\rangle$ as a
subgroup of finite index.

If $G$ preserves the set  $\{[T_+(\phi)],[T_-(\phi)]\}$, then $G$ has
a subgroup of index at most 2 that fixes each of $[T_\pm (\phi)]$  and
hence $G$ is virtually cyclic. Thus we may assume that $G$ does not
preserve $\{[T_+(\phi)],[T_-(\phi)]\}$. So there exists $g\in G$ such
that $[T_{+}(\phi)]g\not \in \{[T_+(\phi)],[T_-(\phi)]\}$ or that
$[T_{-}(\phi)]g\not\in \{[T_+(\phi)],[T_-(\phi)]\}$. We assume the
former as the other case is symmetric. Thus $[T_{+}(\phi)]g\ne
[T_\pm(\phi)]$. Note that $\psi=g^{-1}\phi g\in G$ is also an
atoriodal iwip and that $[T_{+}(\psi)]= [T_{+}(\phi)]g$. We claim that
$[T_{-}(\phi)]\ne [T_\pm(\phi)]g$. Indeed, otherwise we have
$[T_-(g^{-1}\phi g)]=[T_\pm (\phi)]$ and hence $g^{-1}\phi g\in
\Stab_{Out(F_N)}[T_+(\phi)]$ or $g^{-1}\phi g\in
\Stab_{Out(F_N)}[T_-(\phi)]$. In either case (since both stabilizers
contain $\langle\phi\rangle$ as subgroup of finite index)
$g^{-1}\phi^kg=\phi^l$ for some $k\ne 0,l\ne 0$ and therefore
$g^{-1}\phi^k g$ has the same fixed points in $\CVNbar$ as does
$\phi^l$, namely, $[T_\pm (\phi)]$. This contradicts the fact that
$g^{-1}\phi g$ fixes the point $[T_+(\phi)]g\ne [T_\pm(\phi)]$. Thus
$[T_\pm(\phi)], [T_\pm(\psi)]$ are four distinct points. Therefore, by
Corollary~\ref{cor:pingpong}, sufficiently high powers $\phi^M, \psi^M$ freely generate a free subgroup of rank two in $G$, as required. Note that in this case we also have that $[\mu_\pm(\phi)], [\mu_\pm(\psi)]$ are four distinct points by Proposition~\ref{prop:zero}.
\end{proof}

\begin{cor}\label{cor:pair}
Let $\phi,\psi\in \Out(F_N)$ be hyperbolic iwips. Then the following
conditions are equivalent:
\begin{enumerate}
\item The subgroup $\langle \phi,\psi\rangle\subseteq \Out(F_N)$ is not
  virtually cyclic.
\item There exist $m,n\ge 1$ such that $\langle \phi^m, \psi^n\rangle$
  is free of rank two.
\item We have $\langle \phi\rangle \cap \langle \psi\rangle=\{1\}$.
\item The points $[T_\pm(\phi)]$, $[T_\pm(\psi)]$ are four distinct
  elements in $\CVNbar$.
\item The points $[\mu_\pm(\phi)], [\mu_\pm(\psi)]$ are four distinct
  elements in $\PCurr$.
\end{enumerate}
\end{cor}

\begin{proof}
It is obvious that (2) implies (1) and that (2) implies (3).

Suppose that (3) holds. We claim that $[T_\pm (\phi)]$, $[T_\pm
(\psi)]$ are four distinct points in $\CVNbar$. Indeed, suppose that
one of $[T_\pm (\phi)]$ is equal to one of $[T_\pm
(\psi)]$. This means that both $\phi$ and $\psi$ have a common fixed
point in $\CVNbar$ which is a pole of a hyperbolic iwip. The
$\Out(F_N)$-stabilizer of that point is virtually infinite cyclic by
Proposition~\ref{prop:stab}, which implies that $\phi$ and $\psi$ have
some equal nonzero powers, contradicting the assumption $\langle
\phi\rangle \cap \langle \psi\rangle=\{1\}$. Thus $[T_\pm (\phi)]$, $[T_\pm
(\psi)]$ are four distinct points in $\CVNbar$. Therefore by
Corollary~\ref{cor:pingpong} below, sufficiently high powers $\phi^M, \psi^M$
freely generate a free subgroup of rank two in $\Out(F_N)$, so that
(2) holds. Thus (3) implies (2) and, therefore (2) is equivalent to (3).

Suppose now that (1) holds. Suppose that (3) fails and there exist
nonzero $n,m$ such that $\phi^n=\psi^m$. Since $\phi$ has the same
fixed points in $\CVNbar$ as $\phi^n$ and since $\psi$ has the same
fixed points in $\CVNbar$ as $\psi^n$, it follows that $\phi$ and
$\psi$ have a common fixed point in $\CVNbar$. Therefore $\langle
\phi,\psi\rangle$ is contained in the $\Out(F_N)$-stabilizer of that
point, which, by Proposition~\ref{prop:stab}, is virtually
cyclic. This implies that $\langle
\phi,\psi\rangle$ is virtually cyclic, contrary to our assumption
(1). Thus (1) implies (3), which shows that (1), (2) and (3) are
equivalent.

The fact that (4) and (5) are equivalent follows from
Proposition~\ref{prop:zero}. The proof that (3) implies (2) above also
shows that (3) implies (4). The fact that (4) implies (2) follows from
Corollary~\ref{cor:pingpong}. This completes the proof.
\end{proof}

\begin{cor}\label{cor:4}
Let $\phi\in \Out(F_N)$ be a hyperbolic iwip and let $[T_\pm]$ be the attracting and repelling fixed points of $\phi$ in $\CVNbar$.
Then $\Stab_{Out(F_N)}([T_+])=\Stab_{Out(F_N)}([T_-])$.
\end{cor}
\begin{proof}
Suppose there exists $g\in \Stab_{Out(F_N)}([T_+])$ such that $g\not\in \Stab_{Out(F_N)}([T_-])$. Thus $[T_-]g\ne [T_-]$. Since $g$ fixes $T_+$, and since $g$ is a homeomorphism of $\CVNbar$, it follows that $[T_-]g\ne [T_+]$ and hence $\{[T_-],[T_+]\}g\ne \{[T_-],[T_+]\}$.
Put $G=\langle g,\psi\rangle$. Then $G$ does not leave the set $\{[T_-],[T_+]\}$ invariant and hence $G$ contains a free subgroup of rank two by Proposition~\ref{prop:dl} and by Corollary~\ref{cor:3}. On the other hand $G\subseteq \Stab_{Out(F_N)}([T_+])$ and hence $G$ is virtually cyclic, yielding a contradiction.

Thus $\Stab_{Out(F_N)}([T_+])\subseteq \Stab_{Out(F_N)}([T_-])$ and hence, by symmetry, we also have $\Stab_{Out(F_N)}([T_-])\subseteq \Stab_{Out(F_N)}([T_+])$.

Therefore $\Stab_{Out(F_N)}([T_+])=\Stab_{Out(F_N)}([T_-])$, as required.
\end{proof}
Note that, in Corollary~\ref{cor:4}, if $[\mu_\pm]$ are the fixed points of $\phi$ in $\PCurr$, then Corollary~\ref{cor:4} and Proposition~\ref{prop:stab} imply that we in fact have
\[
\Stab_{Out(F_N)}([T_+])=\Stab_{Out(F_N)}([\mu_-])=\Stab_{Out(F_N)}([T_-])=\Stab_{Out(F_N)}([\mu_+]).
\]

\section{North-South Dynamics,  standard neighborhoods and height functions}\label{NS}

\begin{conv}\label{conv:phi}
For the remainder of this section, unless specified otherwise, let $\phi\in \Out(F_N)$, where $N\ge 3$, be a hyperbolic iwip and let $[\mu_+], [\mu_-]\in \PCurr$, $[T_+],[T_-]\in \CVNbar$ and $\lambda_+>1$, $\lambda_->1$ be as in Proposition~\ref{prop:rm} and Proposition~\ref{prop:LL}.

Throughout this section we fix the (arbitrarily chosen) non-projectivized representatives $T_+$ of $[T_+]$,  $T_-$ of $[T_-]$,  $\mu_+$ of $[\mu_+]$ and $\mu_-$ of $[\mu_-]$. 

\end{conv}

\begin{defn}[Standard Neighborhoods]\label{defn:nbhds}
Let
\[
U_+=\{[\mu]\in \PCurr: \langle T_-,\mu\rangle < \langle T_+,\mu\rangle\}
\]
and
\[
U_-=\{[\mu]\in \PCurr: \langle T_-,\mu\rangle > \langle T_+,\mu\rangle\}.
\]
\end{defn}

Note that by Proposition~\ref{prop:zero} for any $[\mu]\ne [\mu_\pm]$ we have $\langle T_+,\mu\rangle>0$ and $\langle T_-,\mu\rangle>0$. Therefore the following function is well-defined:

\begin{defn}[Height function]\label{defn:height}
Define
\[
f: \PCurr-\{[\mu_+],[\mu_-]\}\to\mathbb R
\]
as
\[
f([\mu]):=\log \frac{\langle T_+,\mu\rangle}{\langle T_-,\mu\rangle}.
\]
\end{defn}

It is clear that $f: \PCurr-\{[\mu_+],[\mu_-]\}\to\mathbb R$ is continuous and, moreover, if we put $f([\mu_+])=\infty$ and $f([\mu_-])=-\infty$, this gives a continuous extension of $f$ to $\overline f:\PCurr\to\mathbb R\cup\{\pm \infty\}$. We call $\overline f$ \emph{the extended height function}.
Note that $U_+=(\overline f)^{-1}((0,\infty])$ and $U_-=(\overline f)^{-1}([-\infty,0))$.

\begin{lem}\label{lem:shift}
For any $[\mu]\ne [\mu_\pm]$ we have
\[
f(\phi[\mu])=f([\mu])+\log(\lambda_+\lambda_-).
\]
\end{lem}
\begin{proof}
Let $[\mu]\ne [\mu_\pm]$. Then
\begin{gather*}
f(\phi[\mu])=\log \frac{\langle T_+,\phi\mu\rangle}{\langle T_-,\phi\mu\rangle}=\log \frac{\langle T_+\phi,\mu\rangle}{\langle T_-\phi,\mu\rangle}=\\
=\log \frac{\langle\lambda_+ T_+,\mu\rangle}{\langle\lambda_-^{-1} T_-,\mu\rangle}=\log \big(\lambda_+\lambda_-\frac{\langle T_+,\mu\rangle}{\langle T_-,\mu\rangle}\big)=f([\mu])+\log(\lambda_+\lambda_-).
\end{gather*}
\end{proof}

The continuity of the intersection form and Proposition~\ref{prop:zero} imply that $U_+,U_-$ are disjoint open subsets of $\PCurr$ and that $[\mu_+]\in U_+$ and $[\mu_-]\in U_-$.

\begin{lem}\label{lem:invariant}
We have $\phi(U_+)\subseteq U_+$ and $\phi^{-1}(U_-)\subseteq U_-$.
\end{lem}
\begin{proof}
Let $[\mu]\in U_+$ so that $\langle T_-,\mu\rangle < \langle T_+,\mu\rangle$.
We have $\langle T_-,\phi\mu\rangle=\langle T_-\phi,\mu\rangle=\langle \frac{1}{\lambda_-}T_-,\mu\rangle=\frac{1}{\lambda_-}\langle T_-,\mu\rangle$.
Similarly, $\langle T_+,\phi\mu\rangle=\langle T_+\phi,\mu\rangle=\lambda_+\langle T_+,\mu\rangle$.
Hence
\[
\langle T_-,\phi\mu\rangle=\frac{1}{\lambda_-}\langle T_-,\mu\rangle\le  \langle T_-,\mu\rangle < \langle T_+,\mu\rangle\le \lambda_+\langle T_+,\mu\rangle=\langle T_+,\phi\mu\rangle
\]
so that $[\phi\mu]\in U_+$ by definition of $U_+$. The proof that $\phi^{-1}(U_-)\subseteq U_-$ is symmetric.
\end{proof}

Note that Lemma~\ref{lem:invariant} implies that $\phi(\overline{U}_+)\subseteq \overline{U}_+$ and $\phi^{-1}(\overline{U}_-)\subseteq \overline{U}_-$.

\begin{lem}\label{lem:open}
We have:
\begin{enumerate}
\item $\cap_{n=1}^{\infty} \phi^n(\overline{U}_+)=\{[\mu_+]\}$ and $\cap_{n=1}^{\infty} \phi^{-n}(\overline{U}_-)=\{[\mu_-]\}$.
\item For any neighborhood $V$ of $[\mu_+]$ there exists $n\ge 1$ such that $\phi^n(U_+)\subseteq V$ and for any neighborhood $V'$ of $[\mu_-]$ there exists $n\ge 1$ such that $\phi^{-n}(U_-)\subseteq V'$.
\end{enumerate}
\end{lem}
\begin{proof}

(1) Since $[\mu_+]$ is fixed by $\phi$, it follows that $[\mu_+]\in \phi^n(U_+)$ for every $n\ge 1$ so that $[\mu_+]\in  \cap_{n=1}^{\infty} \phi^n{U_+}$. Suppose there exists $[\mu]\in \cap_{n=1}^{\infty} \phi^n\overline{U}_+$ such that $[\mu]\ne [\mu_+]$. Let $[\mu_n]=\phi^{-n}[\mu]$ for $n\ge 1$. Thus $[\mu_n]\in \overline{U}_+$ for every $n\ge 1$. On the other hand $[\mu]\ne [\mu_+]$ implies that $\lim_{n\to\infty} \phi^{-n}[\mu]=[\mu_-]$. Since $U_-$ is open neighborhood of $[\mu_-]$, there exists $n_0\ge 1$ such that $[\mu_{n_0}]=\phi^{-{n_0}}[\mu]\in U_-$. Thus $\mu_{n_0}\in \overline{U}_+\cap U_-$ which contradicts the fact that $\overline{U}_+\subseteq\{ [\nu]\in \PCurr: \langle T_-,\nu\rangle \le \langle T_+,\nu\rangle\}$ and $U_-$ are disjoint.
This establishes part (1) of the lemma.

(2) Let $V$ be an open neighborhood of $[\mu_+]$ and suppose that there is no $n\ge 1$ such that $\phi^n(U_+)\subseteq V$. Then there exists a sequence $[\mu_n]\in \phi^n(U_+)-V$. After passing to a subsequence, we have $[\mu_{n_i}]\to [\mu]$ as $i\to\infty$. Since $\PCurr-V$ is closed, we have $[\mu]\not\in V$. Since $[\mu_{n_i}]\in \phi^{n_i}(\overline{U}_+)$ and since the closed sets $\phi^{n}(\overline{U}_+)$ are nested, it follows that $[\mu]\in \cap_{n\ge 1} \phi^{n}(\overline{U}_+)$. Therefore by part (1) we have $[\mu]=[\mu_+]$, which contradicts the fact that $[\mu]\not\in V$.
\end{proof}

\begin{lem}\label{lem:0}
For any neighborhoods $U$ of $[\mu_+]$ and $V$ of $[\mu_-]$ there exists $M\ge 1$ such that for every $n\ge M$ we have $\phi^n(\PCurr-V)\subseteq U$ and $\phi^{-n}(\PCurr-U)\subseteq V$.
\end{lem}
\begin{proof}
By applying Lemma~\ref{lem:open} and making $U,V$ smaller if necessary, we may assume that $U$ and $V$ are disjoint open neighborhoods of $[\mu_+]$ and $[\mu_-]$ accordingly such that $\phi(U)\subseteq U$ and $\phi^{-1}(V)\subseteq V$. Let $K=\PCurr-(U\cup V)$. Thus $K$ is a compact subset of $\PCurr$. Since $[\mu_-]\not\in K$, part (2) of Proposition~\ref{prop:rm} implies that
\[
K\subseteq \cup_{n=1}^\infty \phi^{-n}(U).
\]
Since the sets $\phi^{-n}(U)$ are open and $K$ is compact, there exists $p\ge 1$ such that $K\subseteq \cup_{n=1}^p \phi^{-n}(U_+)$. The assumption that $\phi(U)\subseteq U$ implies that $\cup_{n=1}^p \phi^{-n}(U)=\phi^{-p}(U)$. Thus $K\subseteq \phi^{-p}(U)$, so that $\phi^{p}(K)\subseteq U_+$. Since  $\phi^n(U)\subseteq U$ for every $n\ge 1$, it follows that $\phi^n(K)\subseteq U$ for every $n\ge p+1$ and, obviously $\phi^n(U)\subseteq U$ for every $n\ge p+1$. Hence $\phi^n(\PCurr-V)\subseteq U$ for every $n\ge p+1$, as required. The argument for $\phi^{-n}$ is symmetric.
\end{proof}

\begin{cor}\label{cor:pd}
The group $\langle \phi\rangle$ acts properly discontinuously and co-compactly on $\PCurr-\{[\mu_+],[\mu_-]\}$.
\end{cor}
\begin{proof}
Let $K\subseteq \PCurr-\{[\mu_+],[\mu_-]\}$ be a compact subset. Choose open neighborhoods $U_+$ of $[\mu_+]$ and $U_-$ of $[\mu_-]$ in $\PCurr$ so that the sets $U_+,U_-,K$ are pairwise disjoint. Let $M\ge 1$ be as in Lemma~\ref{lem:0}. Therefore for any $n\in \mathbb Z$ with $|n|\ge M$ we have $K\cap \phi^{n}K=\emptyset$ since $\phi^{n}(K)\subseteq U_+\cup U_-$. This shows that $\langle \phi\rangle$ acts properly discontinuously on $\PCurr-\{[\mu_+],[\mu_-]\}$.

To see that the action of $\langle \phi\rangle$ on $\PCurr-\{[\mu_+],[\mu_-]\}$ is co-compact, put
\[
K:=\{[\mu]\in \PCurr-\{[\mu_+],[\mu_-]\} \mid 0\le f(\mu)\le \log(\lambda_+\lambda_-)\}
\]
where $f$ is the height function as in Definition~\ref{defn:height}. The set $K$ is compact since the extended height function $\overline f:\PCurr\to\mathbb R\cup\{\pm\infty\}$ is continuous and $K=(\overline{f})^{-1}([0, \log(\lambda_+\lambda_-)])$. Lemma~\ref{lem:shift} implies that $\langle \phi\rangle K= \PCurr-\{[\mu_+],[\mu_-]\}$. Thus the action of $\langle \phi\rangle$ on $\PCurr-\{[\mu_+],[\mu_-]\}$ is co-compact, as required.
\end{proof}

We summarize the preceding results in the following:
\begin{prop}\label{prop:summary}
Let $\phi$, $T_+$, $T_-$, $\mu_+$, $\mu_-$ be as in Convention~\ref{conv:phi}. Let $U_+,U_-$ be as in Definition~\ref{defn:nbhds}.
Then the following hold:

\begin{enumerate}
\item $U_+$ is an open neighborhood of $[\mu_+]$ and $U_-$ is an open neighborhood of $[\mu_-]$ in $\PCurr$ such that $\phi(U_+)\subseteq U_+$ and $\phi^{-1}(U_-)\subseteq U_-$.
\item $\cap_{n\ge 1} \phi^n(\overline{U}_+) =\{[\mu_+]\}$ and $\cap_{n\ge 1} \phi^{-n}(\overline{V}_-) =\{[\mu_-]\}$.
\item For any neighborhood $U$ of $[\mu_+]$ there exists $n\ge 1$ such that $\phi^n(U_+)\subseteq U$ and for any neighborhood $U'$ of $[\mu_-]$ there is $n\ge 1$ such that $\phi^{-n}(U_-)\subseteq U'$.
\item For any neighborhoods $U$ of $[\mu_+]$ and $U'$ of $[\mu_-]$ in $\PCurr$ there exists $M\ge 1$ such that for every $n\ge M$ we have $\phi^n(\PCurr-U')\subseteq U$ and $\phi^{-n}(\PCurr-U)\subseteq U'$.
\item The action of $\langle \phi\rangle$ on $\PCurr-\{[\mu_+],[\mu_-]\}$ is properly discontinuous and cocompact.
\end{enumerate}
\end{prop}

A symmetric argument gives analogous statements for neighborhoods of $[T_+]$ and $[T_-]$ in $\CVNbar$.
(The only difference is that in the proofs the result of Reiner Martin about the North-South dynamics for the action of $\phi$ on $\PCurr$ has to be replaced by the corresponding result of Levitt-Lustig about the North-South dynamics for the action of $\phi$ on $\CVNbar$).

We summarize them in the following:

\begin{prop}\label{prop:sum}
Let $\phi$, $T_+$, $T_-$, $\mu_+$, $\mu_-$ be as in Convention~\ref{conv:phi}. Define
\[
V_+=\{[T]\in \CVNbar: \langle T, \mu_-\rangle < \langle T, \mu_+\rangle \},
\]
\[
V_-=\{[T]\in \CVNbar: \langle T, \mu_-\rangle > \langle T, \mu_+\rangle \}
\]

Then the following hold:

\begin{enumerate}
\item $V_+$ is an open neighborhood of $[T_+]$ and $V_-$ is an open neighborhood of $[T_-]$ in $\CVNbar$ such that $V_+\phi\subseteq V_+$ and $V_-\phi^{-1}\subseteq V_-$.
\item $\cap_{n\ge 1} \overline{V}_+\phi^n =\{[T_+]\}$ and $\cap_{n\ge 1} \overline{V}_-\phi^{-n} =\{[T_-]\}$.
\item For any neighborhood $V$ of $[T_+]$ there exists $n\ge 1$ such that $V_+\phi^n\subseteq V$ and for any neighborhood $V'$ of $[T_-]$ there is $n\ge 1$ such that $V_-\phi^{-n}\subseteq V'$.
\item For any neighborhoods $V$ of $[T_+]$ and $V'$ of $[T_-]$ in $\CVNbar$ there exists $M\ge 1$ such that for every $n\ge M$ we have $(\CVNbar-V')\phi^n\subseteq V$ and $(\CVNbar-V)\phi^{-n}\subseteq V'$.
\item The action of $\langle \phi\rangle$ on $\CVNbar-\{[T_+],[T_-]\}$ is properly discontinuous and cocompact.
\end{enumerate}

\end{prop}

\begin{rem}\label{rem:uniform}
Note that Proposition~\ref{prop:sum} implies Proposition~\ref{uniform}, and
gives a new proof of the latter, using only the pointwise nature of
the ``North-South'' dynamics for hyperbolic iwips. This in turn also yields
another proof of Corollary~\ref{cor:pingpong}.
\end{rem}

\begin{lem}\label{lem:1}
Let $\phi$, $\mu_+$, $\mu_-$ be as in Convention~\ref{conv:phi}, and let 
$A$ be a free basis of $F_N$. Then there exist an open set $U\subseteq\PCurr$ containing $[\mu_+]$ and an integer $M_0\ge 1$ with the following property:

For every $[\mu]\in U$ and every $n\ge M_0$ we have $\langle T_A, \phi^{n}\mu\rangle\ge 2\langle T_A,\mu\rangle$.

\end{lem}

\begin{proof}

Let $f:\Gamma\to \Gamma$ be a train-track representative of $\phi$. There is an integer $L_1\ge 1$, an integer $M_1\ge 1$ and a constant $\lambda_1> 1$ such that the following holds. If $\gamma=\gamma_1\gamma_2\gamma_3$ is a reduced concatenation of three reduced paths in $\Gamma$ where $\gamma_2$ is legal of simplicial length $\ge L_1$ then for any $n\ge M_1$ in the cancellation between the tightened forms of $f^n\gamma_1$, $f^n\gamma_2$, $f^n\gamma_3$, a segment of $f^n\gamma_2$ of simplicial length at least $\lambda_1|\gamma_2|$ survives.

Let $U$ be the neighborhood of $[\mu_+]$ defined by the condition that for $[\mu]\in U$ the frequencies with respect to the simplicial metric on $\Gamma$ of all legal paths of simplicial length $L_1$ in $[\mu]$ are almost the same as they are for the corresponding frequencies in $[\mu_+]$ and the frequencies of illegal turns are sufficiently close to zero.  Here for a reduced path $v$ in $\Gamma$ and a nonzero current $\mu$ by a \emph{frequency} of $v$ in $\mu$ we mean the ratio $\frac{\langle v,\mu\rangle_\Gamma}{\langle T_\Gamma, \mu\rangle}$. It is easy to see that this frequency depends only on $T_\Gamma$ and the projective class $[\mu]$ of $\mu$. See \cite{Ka2} for more details.  Then for any rational current $[\eta_w]\in U$, where $w$ is a reduced cyclic path in $\Gamma$, when we write $w$ as a concatenation of maximal legal segments, the legal segments of length $\ge L_1$ each will constitute at least $1/2$ of the simplicial length of $w$. Then for any $k\ge 1$ there is $M\ge 1$ independent of $w$ such that $|[f^nw]|\ge 2^k|w|$ for every $w$ as above and every $n\ge M$. Therefore for every $[\mu]\in U$ we have $\langle T_\Gamma, \phi^n\mu\rangle \ge 2^k \langle T_\Gamma,\mu\rangle$ for every $n\ge M$. Here $T_\Gamma$ is the universal cover of $\Gamma$ with the simplicial metric. Since the translation length functions on $\FN$ corresponding to the trees $T_A$ and on $T_\Gamma$ are bi-Lipschitz equivalent, it follows that, by choosing a sufficiently large $k$, there exists $M_0\ge 1$ such that for every $n\ge M_0$ and every $[\mu]\in U$ we have $\langle T_A, \phi^{n}\mu\rangle\ge 2\langle T_A,\mu\rangle$, as required.
\end{proof}

\begin{cor}\label{cor:2}
Let $\phi$, $\mu_+$, $\mu_-$ be as in Convention~\ref{conv:phi}, and let  
$A$ be a free basis of $F_N$.

Then for any neighborhood $V$ of $[\mu_-]$ in $\PCurr$
there exists an integer $M_1=M_1(V,\phi)\ge 1$ such that for every
$[\mu]\in \PCurr-V$, for every
$[T]\in \CVNbar-U$  and any $n\ge M_1$ we have $\langle
T_A, \phi^n\mu\rangle\ge 2\langle T_A,\mu\rangle$.
\end{cor}
\begin{proof}

Let $U$ and $M_0$ be provided by Lemma~\ref{lem:1}. By making $V$ smaller we may assume that $V$ is disjoint from $U$.
By making $U$ smaller, via application of Lemma~\ref{lem:invariant} and Lemma~\ref{lem:open} we may assume that, in addition to the conditions guaranteed by Lemma~\ref{lem:1}, we have $\phi(U)\subseteq U$. Note that since $\phi(U)\subseteq U$, then for every $[\nu]\in U$ and every integer $k\ge 1$ we have $\langle T_A, \phi^{kM_0}\nu\rangle \ge 2^k \langle T_A, \nu\rangle$.

By Lemma~\ref{lem:0} there exists $M\ge 1$ such that $\phi^{n}(\PCurr-V)\subseteq U$ for every $n\ge M$.
Since $M$, $M_0$ are fixed, there exists $C>0$ such that for every $\mu\in \Curr(F_N)$ we have $\langle T_A,\phi^{M}\mu\rangle\ge C\langle T_A,\mu\rangle$. Let $k\ge 1$ be an arbitrary integer.
Let $[\mu]\in \PCurr-V$ be arbitrary. We have $\phi^{n}([\mu])\in U$ for every $n\ge M$.   We also have $\langle T_A,\phi^{M}\mu\rangle\ge C\langle T_A,\mu\rangle$ and $\langle T_A,\phi^{kM_0+M}\mu\rangle \ge 2^k\langle T_A, \phi^{M}\mu\rangle$ by Lemma~\ref{lem:1}. Therefore $\langle T_A,\phi^{kM_0+M}\mu\rangle \ge 2^kC\langle T_A,\mu\rangle$. The statement of the corollary now easily follows.
\end{proof}

\begin{cor}\label{cor:3}
Let $\phi,\psi\in \Out(F_N)$ be hyperbolic iwips such that $\langle\phi,\psi\rangle$ is not virtually cyclic. Then there exists $M\ge 1$ such that for any $n,m\ge M$ the subgroup $\langle \phi^n, \psi^m\rangle$ is free of rank two.
\end{cor}
\begin{proof}
Corollary~\ref{cor:pair} implies that $[T_{\pm}(\phi)]$ and
$[T_{\pm}(\psi)]$ are four distinct points of $\CVNbar$. Therefore the
statement follows from Corollary~\ref{cor:pingpong}.
\end{proof}

Note that by Proposition~\ref{prop:zero} for hyperbolic iwips $\phi,\psi$ the assumption that $[T_{\pm}(\phi)]$, $[T_{\pm}(\psi)]$ are four distinct points of $\CVNbar$ is equivalent to the condition that $[\mu_{\pm}(\phi)]$, $[\mu_{\pm}(\psi)]$ are four distinct points of $\PCurr$.

\begin{lem}\label{lem:3-4}
Let $\phi,\psi\in \Out(F_N)$ be hyperbolic iwips such that $\langle \phi,\psi\rangle\subseteq \Out(F_N)$ is not virtually cyclic. Let $A$ be a free basis of $F_N$. Then there exists $M\ge 1$ with the following property:

For any $\mu \in \Curr(F_N)$ and for any $n,m\ge M$, for at least three out of four elements $\alpha$ of $\{\phi^n, \psi^m,\phi^{-n}, \psi^{-m}\}$ we have
\[
2\langle T_A,\mu\rangle\le \langle T_A,\alpha\mu\rangle.
\]
\end{lem}

\begin{proof}
Since $\langle \phi,\psi\rangle\subseteq \Out(F_N)$ is not virtually cyclic, the four eigencurrents $[\mu_+(\phi)]$, $[\mu_+(\psi)]$, $[\mu_-(\phi)]$, $[\mu_-(\psi)]$ are four distinct elements in $\PCurr$. Take four disjoint open neighborhoods $U_+(\phi)$, $U_+(\psi)$, $U_-(\phi)$, $U_-(\psi)$.

Let $M\ge 1$ be the maximum of the the constants $M_1(U_-(\phi),\phi)$, $M_1(U_-(\psi),\psi)$, $M_1(U_+(\phi),\phi^{-1})$, $M_1(U_+(\psi),\psi^{-1})$ provided by Corollary~\ref{cor:2}. Let $\mu\in \Curr(F_N),\mu\ne 0$ be arbitrary. Then there are at least three of the four sets  $U_+(\phi)$, $U_+(\psi)$, $U_-(\phi)$, $U_-(\psi)$ that $[\mu]$ does not belong to, and the statement of the lemma follows from Corollary~\ref{cor:2}
\end{proof}

Lemma~\ref{lem:3-4} immediately implies:

\begin{cor}\label{cor:3-4}
Let $\phi,\psi\in \Out(F_N)$ be hyperbolic iwips such that $\langle \phi,\psi\rangle\subseteq \Out(F_N)$ is not virtually cyclic. Let $A$ be a free basis of $F_N$. Then there exists $M\ge 1$ with the following property: For any $w\in F_N$ and any $n,m\ge M$, for at least three out of four elements $\alpha$ of $\{\phi^n, \psi^m,\phi^{-n}, \psi^{-m}\}$ we have
\[
2||w||_A\le ||\alpha(w)||_A. 
\]
\end{cor}


Via a standard argument (c.f. the proof of Theorem~5.2 in \cite{BFH97}), Corollary~\ref{cor:3-4} and the Bestvina-Feighn Combination Theorem~\cite{BF92} imply:
\begin{thm}\label{thm:hyperb}
Let $N\ge 3$ and let $\phi,\psi\in \Out(F_N)$ be hyperbolic iwips such that $\langle \phi,\psi\rangle\subseteq \Out(F_N)$ is not virtually cyclic and let $\Phi,\Psi\in \Aut(F_N)$ be such that $\Phi$ represents $\phi$ and $\Psi$ represents $\psi$. There exists $M\ge 1$ such that for any $n,m\ge M$ the subgroup $\langle \phi^n, \psi^m\rangle\subseteq \Out(\FN)$ is free of rank two, every nontrivial element of this subgroup is hyperbolic and the group
\[
G_{n,m}=\langle F_N, t,s| t^{-1}wt=\Phi^n(w), s^{-1}ws=\Psi^m(w) \text{ for every } w\in F_N\rangle
\]
is word-hyperbolic.
\end{thm}

The ``3 out of 4" condition in Corollary~\ref{cor:3-4} was first introduced by Lee Mosher for surface homeomorphisms in~\cite{Mosher} where he used it to construct an example of a (closed surface)-by-(free of rank two) word-hyperbolic group. Similarly,  the  ``3 out of 4"  condition was used by Bestvina, Feign and Handel~\cite{BFH97} to construct a free-by-free word-hyperbolic group.  Our proof of the  ``3 out of 4" condition in Corollary~\ref{cor:3-4}  is different from both the approaches of Mosher and of Bestvina-Feighn-Handel:  our method is based on exploiting North-South dynamics of hyperbolic iwips acting on the space of projectivized currents rather than on the space of laminations.

\begin{rem}\label{purely-hyperbolic}
In Theorem~\ref{thm:hyperb} it is also easy to conclude that 
for every nontrivial element $\theta\in \langle
\Phi,\Psi\rangle\subseteq \Aut(F_N)$ the automorphism $\theta$ is
hyperbolic. This can be seen directly from the Annuli Flare Condition~\cite{BF92}
for the group $G_{n,m}$ above.  Alternatively,  suppose $\theta$ is not
hyperbolic. Then $\theta$ is not atoroidal, that is to say $\theta$
has a periodic conjugacy class. This yields a $\mathbb Z\times\mathbb
Z$-subgroup in $G_{ni,m}$ which contradicts the fact that
$G_{n,m}$ is word-hyperbolic.
\end{rem}

\begin{rem}
The same proof as that of Theorem~\ref{thm:hyperb} shows that the conclusion of this theorem holds if instead of two elements of $\Out(F_N)$ we use $k\ge 2$ hyperbolic iwip elements $\phi_1,\dots, \phi_k$ with the property that for every $1\le i<j\le k$ the subgroup $\langle \phi_i,\phi_j\rangle\subseteq \Out(F_N)$ is not virtually cyclic.
\end{rem}

\section{Ping-pong for Schottky type groups and its consequences}

\begin{defn}\label{ping-pong0}
Let $G$ be a group that acts on a non-empty set  $X$ (either on the
left or on the right).  Suppose that the group  $G$  is generated by
two specified elements $g$  and  $h$, and  that $X$ contains pairwise
disjoint {\em geographical} nonempty subsets  $N$ (= ``North''), $S$ (= ``South''), $E$ (= ``East'') an $W$ (= ``West''),
such that $g$ maps $X - S$ into $N$,  $g^{-1}$  maps  $X - N$  into
$S$, and similarly  $h$  with  $E$  and  $W$. We call this action of
$G$ on $X$ a {\em 2-generator ping-pong action} with respect to $g$ and $h$.
\end{defn}

For the rest of this section, unless specified otherwise we assume
that we are given a 2-generator ping-pong action of $G$ on $X$ with
respect to $g$ and $h$.

It follows from Felix Klein's classical argument that in the above
situation $G$ is free with a free basis $\{g, h\}$. Notice that we
purposefully did not specify whether  $G$  acts from the left or from
the right on  $X$.  For any reduced and cyclically reduced word  $w =
x_1 \ldots x_q$  in  $\{g,h\}^{\pm 1}$ we define the {\em final acting letter} to be the $x_i$ that acts last on $X$. Thus, if we have a left action, then the final acting letter of  $w$ is the first letter $x_1$, and in case of a right action it is the last letter $x_q$.

\smallskip

We define the  {\em forward limit region} of  $w$ (reduced and cyclically reduced) as the nested intersections of the images of  $Y$  under  $w^n$,  for  positive  $n$, where  $Y = N$  if  $w$  has the final acting letter  $g$, $Y = S$  if  $w$  has the final acting letter  $g^{-1}$, and similarly  $Y = E$  or  $Y = W$  if  $w$  has the final acting letter  $h$  or  $h^{-1}$.

\begin{rem}
\label{fixed-points}
It follows directly from the above set-up that for a cyclically
reduced $w$, any fixed point of $w$ in $X$ must be contained either in the forward limit region of $w$ or in that of $w^{-1}$ (one could call the latter the {\em backward limit region} of $w$).
\end{rem}

\begin{defn}
\label{ping-pong}
A 2-generator ping-pong action of $G$ on $X$ as above is called {\em open} if in addition
$X$ is a topological space, $G$ acts by homeomorphisms on $X$, and the geographical subsets $N, S, E$ and $W$ are open.
\end{defn}

Notice also that the restriction of any 2-generator ping-pong action
to a $G$-invariant subset $X' \subseteq X$ is also 2-generator ping-pong: One simply redefines the set North as $X' \cap N$, South as $X' \cap S$, etc. Of course, if the action on $X$ is open, then so is the action on $X'$.

\smallskip

Suppose now that $\phi,\psi\in \Out(\FN)$ are hyperbolic iwips such that
$\langle \phi, \psi\rangle$ is not virtually cyclic. We already know by the
results of Section~\ref{NS} that there is $M\ge 1$ such that for every
$n, m\ge M$ the actions of $G=\langle \phi^n,\psi^m\rangle$ on both,
$\CVNbar$ (from the right) and on $\Pr\Curr(\FN)$ (from the left) are
open 2-generator ping-pong actions with respect to $\phi^n,\psi^m$. Note
that in this case for the left action of $G$ on $\Pr\Curr(\FN)$, the
``north'' set $N$ contains $[\mu_+(\phi)]$, the ``south'' set $S$
contains $[\mu_-(\phi)]$, the ``east'' set $E$ contains $[\mu_+(\psi)]$ and
the ``west'' set $W$ contains $[\mu_-(\psi)]$. Similarly, for the right
action of $G$ on $\CVNbar$ the geographical sets $N$, $S$, $E$, $W$
contain $[T_+(\phi)]$, $[T_-(\phi)]$, $[T_+(\psi)]$, and $[T_-(\psi)]$
respectively.

The main goal of this section is to prove the following:

\begin{prop}
\label{crucial-lemma}
Let  $\phi, \psi \in\Out(\FN)$ be hyperbolic iwips, and assume that the actions of $G =
\langle \phi, \psi\rangle$ on both,  $\CVNbar$  (from the right)  and on
$\Pr\Curr(\FN)$  (from the left) are open 2-generator ping-pong
actions with respect to $\phi,\psi$.

Then there exist constants  $m_0, n_0 \geq 1$ with the following
property.  Suppose $m \geq m_0$, $n \geq n_0$ and $w$ is a cyclically
reduced word in  $\phi^{\pm 1}, \psi^{\pm 1}$  which starts in  $\phi^m$  and
ends in  $\phi^n$, and suppose $[\mu]\in \PCurr$ is such that
$[\mu]$ is contained in the forward limit region of $w$ and such
for some $\lambda \ne 1$ we have $w \mu = \lambda \mu$.

Then $\lambda>1$.
\end{prop}

For the remainder of this section we suppose that the assumptions of
Proposition \ref{crucial-lemma} are satisfied.
In order to prove this proposition we need first some preliminary considerations.

\smallskip

Recall that $\CVN=\mathbb P \cvn$ is the projectivization of $\cvn$.
We consider the (right) action of $G$ on  $\CVNbar$. Denote by $cv^1_N
\subseteq\cvn$ the lift of  $\CVN$ to $\cvn$  where every tree has covolume 1. Note that $cv^1_N$ is invariant under the action of $\Out(\FN)$ on $\cvn$, and that the group $G = \langle \phi, \psi \rangle$ acts on $cv^1_N$ again as open 2-generator ping-pong group, since by the hypothesis of Proposition \ref{crucial-lemma} it does so on $\CVNbar$ and thus on $\CVN$.

\smallskip

\begin{conv}\label{conv}
Let $\phi\in \Out(F_N)$ be a hyperbolic iwip.
Let  $[\mu_+]=[\mu_+(\phi)]$  be an expanding fixed projectivized current of
$\phi$, i.e.  $\phi \mu_+ = \lambda_+ \mu_+$, with  $\lambda_+ > 1$.  Let
$[T_-]$ be a contracting fixed projectivized $\R$-tree of $\phi$, i.e.
$T_- \phi = \lambda_-\inv T_-$, with  $\lambda_- > 1$.  We recall  that,
by Proposition~\ref{prop:rm} and Proposition~\ref{prop:LL}, both
$\mu_+$ and $T_-$ are uniquely determined by $\phi$ up to scalar
multiplication, and that $\langle T_-, \mu_+ \rangle = 0$. Moreover, by
Proposition~\ref{prop:zero}, up to scalar multiples, $T_-$ is the unique
tree  $T \in \cvnbar$ satisfying $\langle T, \mu_+ \rangle = 0$.

Similarly, let $[\mu_-]=[\mu_-(\phi)]$ and $[T_+]=[T_+(\phi)]$ be the
contracting fixed projectivized current and expanding fixed
projectivized tree for $\phi$. Note that $T_+\phi=\lambda_+T_+$ and $\phi^{-1}\mu_-=\lambda_-\mu_-$.
\end{conv}

\smallskip

For any subset $V \subseteq\cvnbar$ we denote by $\Pr V \subseteq\CVNbar$ its canonical image under projectivization. Note that the closure $\overline V$ of $V$ in $\cvnbar$ projects to a subset $\Pr \overline V \subseteq\CVNbar$ which is contained in the closure $\overline{\Pr V}$ of $\Pr V$ in $\CVNbar$, but that in general the two are not equal:
\[\Pr \overline V \subseteq \overline{\Pr V}.\]

\begin{lem}\label{(8')halfbound}
Let  $V$  be a subset of $cv^1_N$  with the property that
$[T_-]\not\in\overline{\Pr V}$. Suppose also that
for some constant  $c > 0$ we have
\[\langle T, \mu_+ \rangle\,\, \ge c\]
for all  $T$  in  $V$.  Then for some sufficiently small
neighborhood  $U \subseteq\Curr(\FN)$  of  $\mu_+$ one has
\[\langle T, \mu \rangle\,\,\,  \geq \, \,  \frac{c}{2}\]
for any  $T$  in  $V$ and for any  $\mu \in U$.
\end{lem}

\begin{proof}
Suppose that the statement of the lemma is false.
Then there exist a sequence of currents  $\mu_i \in \Curr(\FN)$  converging to  $\mu_+$, and
a sequence of trees  $T_i \in V$, which all satisfy
\[\langle T_i, \mu_i \rangle \,\, < \,\, \frac{c}{2} \,\, \leq \,\, \frac{\langle T_i, \mu_+ \rangle}{2}.\]
By compactness of  $\CVNbar$  we can extract a subsequence
of the  $T_i$, which we still denote $T_i$,  which converges
projectively to some $\R$-tree  $T_\infty$  in $\cvnbar$. Let
$\lambda_i > 0$ be such that $\lim_{i\to\infty}\lambda_i T_i=T_\infty$
in $\cvnbar$.

By continuity and linearity of the intersection form we have:

\begin{gather*}\langle T_\infty, \mu_+ \rangle = \langle \lim \lambda_i T_i, \lim \mu_i \rangle = \lim \lambda_i \langle T_i,  \mu_i \rangle\\
\leq
\lim  \frac{\lambda_i \langle T_i,  \mu_+ \rangle}{2} =
\frac{\langle \lim \lambda_i T_i,  \mu_+ \rangle}{2} =
\frac{\langle T_\infty,  \mu_+ \rangle}{2}
\end{gather*}

Hence
$\langle T_\infty, \mu_+ \rangle = 0$ and therefore, by
Proposition~\ref{prop:zero}, $[T_\infty]=[T_-]$.
However, $[T_\infty]$  is contained in the closure of  $\Pr V$ in
$\CVNbar$, which contradicts our assumption that $[T_-]\not\in\overline{\Pr V}$.
\end{proof}

\begin{lem}
\label{several-old}
There exists a non-empty subset $V \subseteq cv^1_N$ with the following properties:

\begin{enumerate}
\item
$[T_-] \notin \overline{\Pr V}$

\item
$c := \inf \{ \langle T, \mu_+ \rangle \mid T \in V \}  > 0$

\item
$V \phi \subseteq V$, and
$\langle T \phi, \mu_+\rangle \,\, = \,\, \lambda_+  \langle T, \mu_+\rangle$ for any  $T \in V$.

\item
There exists a tree $T_0 \in V$ such that
for every reduced word  $w$  in  $\phi^{\pm 1}$   and  $\psi^{\pm 1}$  that does not end in  $\phi^{-1}$, the tree $T_0 w$  lies in  $V$.
\end{enumerate}
\end{lem}

\begin{proof}

Let $N,S,E,W\subseteq \CVNbar$ be the geographical subsets for the
ping-pong action of $G=\langle \phi,\psi\rangle$ on $\CVNbar$. Recall that
$N$ and $S$ are open neighborhoods of $[T_+]=[T_+(\phi)]$ and
$[T_-]=[T_-(\phi)]$ accordingly. Choose an arbitrary $T_0\in cv^1_N$ such
that $[T_0]\not\in N \cup S \cup E \cup W$.

Put
\[
V=\{T\in T_0 \Out(F_N): [T]\not\in S\}.
\]

We claim that these choices of $V$ and $T_0$ satisfy all the requirements
of the lemma. Indeed, conditions (1), (3) and (4) follow immediately
from the definitions of $V$ and $T_0$ and from the ping-pong
properties of the action of $G$ on $\CVNbar$. Suppose that condition
(2) fails. Then there exists a sequence $\alpha_n\in \Out(F_N)$ such
that $\lim_{n\to\infty} \langle T_0\alpha_n,\mu_+\rangle=0$. After
passing to a further subsequence, we may assume that
$\lim_{n\to\infty} [T_0\alpha_n]=[T_\infty]$ and that
$\lim_{n\to\infty} c_nT_0\alpha_n=T_\infty$ for some $T_\infty\in
\cvnbar$ and some $c_n>0$. Note that $[T_\infty]$ belongs to the
closure of $\mathbb PV$ and hence $[T_\infty]\ne [T_-]$. Note that
because $T_0\in cv^1_N$, we have $\langle
T_0\alpha_n,\mu_+\rangle>0$. After passing to a further subsequence we
may also assume that all $\alpha_n$ are distinct and that
\[
0<\langle T_0\alpha_n,\mu_+\rangle\le 1
\]
for all $n\ge 1$. Since $\alpha_n$ are distinct and the action of
$\Out(F_N)$ on $cv^1_N$ is properly discontinuous, it follows that
$[T_\infty]\in \partial \CVN=\CVNbar-\CVN$. This in turn implies that
$\lim_{n\to\infty} c_n=0$ (see \cite{KL3} for details).

By the linearity of the intersection form we have:
\[
\langle c_nT_0\alpha_n,\mu_+\rangle=c_n\langle
T_0\alpha_n,\mu_+\rangle\le c_n\to_{n\to\infty} 0
\]
Hence, by continuity, $\langle T_\infty, \mu_+\rangle=0$ which, by
Proposition~\ref{prop:zero}, implies that $[T_\infty]=[T_-]$. This
contradicts our earlier conclusion that $[T_\infty]\ne [T_-]$.
\end{proof}

\begin{lem}
\label{(8)old}
There exists a tree $T_0 \in cv^1_N$, a constant $n_0 \geq 0$ and a neighborhood $U \subseteq\Curr(\FN)$ of $\mu_+$, such that for any reduced word  $w$  in  $\phi^{\pm 1}$   and  $\psi^{\pm 1}$  that does not end in  $\phi^{-1}$, for any  $n \geq n_0$, and for any  $\mu \in U$, one has:
$$
\langle T_0, w \phi^n \mu \rangle  \,\, > \,\,   \langle T_0, \mu \rangle.
$$
\end{lem}

\begin{proof}
Consider a set $V \subseteq cv^1_N$ and a tree $T_0 \in V$ as in Lemma \ref{several-old}.  By definition we have $\lambda_+ > 1$, and by Lemma \ref{several-old}(2), the infimum $c$ of all $\langle T, \mu_+ \rangle $, for any $T \in V$, satisfies $c > 0$.  Thus we can pick $n_0 \geq 0$ so that
$\lambda_+^{n_0} c > 100\langle T_0, \mu_+ \rangle$. From the continuity of the intersection form we deduce
$$\frac{\lambda_+^n c}{2}
> \langle T_0, \mu \rangle$$ for all $n \geq n_0$ and for any $\mu$ in some sufficiently small neighborhood $U_0$ of $\mu_+$.

By part (4) of Lemma~\ref{several-old} the tree $T_0 w$ is contained
in $V$, and hence $T_0 w \phi^n$ is contained in $V \phi^n$. From parts (2)
and (3) of Lemma~\ref{several-old} we deduce that every $T \in V \phi^n$ satisfies:
$$\langle T, \mu_+ \rangle \geq \lambda_+^n c$$
Now Lemma \ref{(8')halfbound}, whose hypotheses are guaranteed by
Lemma \ref{several-old} (1) and (2), implies that
$$\langle T, \mu \rangle \geq \frac{\lambda_+^n c}{2}$$
for any $T \in V \phi^n$ and any $\mu$ in a sufficiently small neighborhood $U' \subseteq\Curr(\FN)$ of $\mu_+$.
It follows from our above choice of $n_0$ that
$$
\langle T_0, w \phi^n \mu \rangle  \,\, = \,\,
\langle T_0 w \phi^n, \mu \rangle  \,\, \geq \frac{\lambda_+^n c}{2}
> \langle T_0, \mu \rangle
$$
for any $\mu$ in $U := U' \cap U_0$.
\end{proof}

We now consider the (left) action of $G$ on $\PCurr$, which by the hypotheses of Proposition \ref{crucial-lemma} is again an open 2-generator ping-pong action.

\smallskip

By the North-South dynamics of iwips (see the results of
Section~\ref{NS}) for any sufficiently large $m \geq 1$  the power
$\phi^m$  maps any compact subset of $\PCurr$, that excludes the
repelling fixed point $[\mu_-]$ of $\phi$,  into  $\Pr U \subseteq\PCurr$.
Here $\Pr U \subseteq\PCurr$ means the image of a neighborhood  $U \subseteq\Curr(\FN)$ of the $\phi$-expanding eigencurrent $\mu_+$ as above.  As an easy consequence of the open 2-generator ping-pong assumption of the $G$-action on $\PCurr$ we thus obtain:

\begin{lem}
\label{forward-limit-region}
For every neighborhood $U \subseteq\Curr(\FN)$ of $\mu_+$ there exists
$m_1 \geq 1$ such that for all $m \geq m_1$ and for every  reduced
word  $w$  in  $\phi^{\pm 1}$  and  $\psi^{\pm 1}$  which does not start or end with  $\phi\inv$, the forward limit region of $\phi^m w$ in $\PCurr$ is contained in  $\Pr U$.
\end{lem}

We can now prove Proposition \ref{crucial-lemma}:

\begin{proof}[Proof of Proposition \ref{crucial-lemma}]
Let $U$, $T_0$, $n_0$ and $m_1$ be as given in Lemma \ref{(8)old} and
Lemma \ref{forward-limit-region}, so that both apply to the given word
$w = \phi^m w' \phi^n$, where $w'$ is a reduced word in $\phi^{\pm 1}$ and
$\psi^{\pm 1}$ which does not start or end with  $\phi\inv$, and where $n \geq n_0$ and $m \geq m_0 := \max(m_1, n_0)$ holds.

Recall that by assumptions of Proposition \ref{crucial-lemma}, we are
given a projective current $[\mu]$ in the forward limit region of $w$
such that $[\mu]$ is fixed by $w$, so that
$w \mu=\lambda \mu$ for some $\lambda>0$. We need to prove that $\lambda>1$.

Since $[\mu]$ contained in the forward limit region of $w$ in
$\PCurr$, it follows that $[\mu]$ is contained in the set $\Pr U$, by Lemma \ref{forward-limit-region}. Hence some scalar multiple $\mu'$ of $\mu$ is contained in $U$, and we deduce from Lemma~\ref{(8)old} that:
$$
\langle T_0, w \mu' \rangle=\langle T_0, \phi^m w' \phi^n \mu' \rangle  \,\, > \,\,   \langle T_0, \mu' \rangle
$$
But since $\mu$ is projectively fixed by $w = \phi^m w' \phi^n$, we have:
$$
\langle T_0, w \mu \rangle  \,\, = \,\,   \langle T_0, \lambda \mu \rangle  \,\, = \,\,   \lambda \langle T_0, \mu \rangle
$$
Hence $\lambda\langle T_0,\mu\rangle > \langle T_0,\mu\rangle$.
Note that $T_0$ is a tree with a free simplicial action of $\FN$, which implies $\langle T_0, \mu \rangle \neq 0$. Therefore
$$
\lambda > 1 \, ,
$$
which proves Proposition \ref{crucial-lemma}.
\end{proof}

An interesting feature of the above proof is that it needs the
ping-pong property of the action of $G$ on both spaces, $\CVNbar$ and
$\PCurr$. As one of these actions is a left action and the other one a
right action, the words $w$ considered must both, start and end in
large powers of $\phi$.  However, this is not a problem for the
application in the next section, since the property of an automorphism
to be an iwip is invariant under conjugation in $\Out(F_N)$.

Recall that $\phi$, $\mu_+=\mu_+(\phi)$ and $\mu_-=\mu_-(\phi)$ are as
in Convention~\ref{conv}.
We finish this section with a lemma that will be used in the next section:

\begin{lem}
\label{strictly-bigger}
There exists neighborhoods $U_+(\phi) \subseteq \Curr(\FN)$ of $\mu_+$ and $U_-(\phi) \subseteq \Curr(\FN)$ of $\mu_-$ such that
\[\langle T, \mu'_+ \rangle + \langle T, \mu'_- \rangle > 0\]
for any $T \in \cvnbar$ and any $\mu'_+ \in U_+, \mu'_- \in U_-$.
\end{lem}

\begin{proof}
It suffices to recall that, by Proposition~\ref{prop:zero}, the above inequality is true for $\mu'_+ = \mu_+$ and $\mu'_- = \mu_-$, and to use the continuity of the intersection form.
\end{proof}


\section{Every Schottky group contains a rank-two free iwip subgroup}

We first prove a property that characterizes those hyperbolic
automorphisms of $\FN$ which are not iwips.

\begin{prop}
\label{poles1}
For every $\Phi \in \Aut(\FN)$ which is hyperbolic but not an iwip
there exist $k\ge1 $, a tree $T_0 \in \cvnbar$ and currents $\nu_+,\nu_- \in \Curr(\FN) \smallsetminus \{0\}$ with the following properties:
\begin{enumerate}
\item
$\Phi^k(\nu_+) = \rho_+ \nu_+$ and $\Phi^k(\nu_-) = \rho_-\inv \nu_- \, ,$ for some $\rho_+, \rho_- > 1$.

\item $\langle T_0, \nu_+ \rangle \, \,  = \, \,  \langle T_0, \nu_- \rangle \, \,  = 0$.
\end{enumerate}
\end{prop}

\begin{proof}
The $\R$-tree  $T_0$ is constructed as described in \cite{GJLL} from a relative train track representative (in the sense of \cite{BH92}) $f : \tau \to \tau$ of $\Phi$ using a (row) eigenvector of the transition matrix $M(f)$ which has non-zero coefficients only for the top stratum of $\tau$. The translation length with respect to $T_0$ satisfies
$$|| w ||_{T_0} = 0$$
for all conjugacy classes $[w] \subseteq \FN$ which are represented by a loop in $\tau$ that does not traverse any edge of the top stratum.  Note that $T_0$ is projectively $\Phi$-invariant, but that the stretching factor may well be equal to 1, in which case $T_0$ is simplicial.

Since $\Phi$ has a reducible power, we can choose a proper free factor $U$ of $\FN$ which is (up to
conjugation) fixed by some power $\Phi^k$ with $k \geq 1$, and which
does not contain properly any non-trivial free factor of $\FN$ that is
$\Phi^h$-invariant (up to conjugation) for any $h \geq 1$. After
composing $\Psi=\Phi^k$ with an inner automorphism if necessary, we
may assume that $\Psi(U)=U$. Note that $U$ is not cyclic since by assumption $\Phi$ is hyperbolic.

We claim that, moreover, $U$ can be
chosen in such a way that $U$ fixes a point of $T_0$. Indeed, since
every maximal elliptic subgroup for the above tree $T_0$ is (up to conjugation) a
$\Phi$-invariant free factor of $F_N$, if such a subgroup is non-trivial, it
must contain the $\Phi$-orbit of some free factor $U$ as above. It
thus remains to argue that there exist at least one elliptic subgroup of
$T_0$ which is non-trivial.  If $T_0$ is simplicial, this is obvious,
as otherwise the action of $\FN$ on $T_0$ would be free, and since
$T_0$ is $\Phi$-invariant, the automorphism $\Phi$ would be periodic
(up to conjugation) and hence not hyperbolic.  In the complementary
case, where $T_0$ is not simplicial, the top stratum of the above
train track map $f: \tau \to \tau$ must be exponentially growing.
Thus, if $\tau$ has more than one stratum, every lower stratum
contributes to the elliptic subgroups of $T_0$, and hence the bottom
stratum would define a non-trivial elliptic subgroup for $T_0$, as
required. Finally, if there is only one stratum in $\tau$, then either
$\Phi$ was an iwip, or else, by Proposition~5.1 of \cite{Lu3} (see also chapter~7 of~\cite{Lu2}),
one of the vertices of
$\tau$ can be blown-up to give a new $\Phi$-invariant train track with
a periodic top stratum. For this new train track $\tau'$ its stable
tree $T_0'$ will be simplicial (as explained in detail in the proof of Proposition~5.1 of \cite{Lu3}),
and the previous arguments would
apply.  Thus indeed $U$ and $T_0$ can be chosen so that $U$ fixes a
point in $T_0$.

By the choice of $U$, the restriction $\Psi|_U\in
Aut(U)$ is an iwip automophism of $U$. Moreover, $\Psi|_U$ has no
periodic conjugacy classes (since $\Phi$ has no periodic conjugacy
classes), so that $\Psi|_U\in \Aut(U)$ is also atoroidal.

Therefore there exist a projectively unique
$\Psi|_U$-invariant  expanding current $\mu^U_+$ and  a projectively unique
$\Psi|_U$-invariant  contracting current $\mu^U_-$ in $\Curr(U)$. More
precisely, $\Psi|_U(\mu^U_+)=\rho_+ \mu^U_+$ and
$(\Psi|_U)^{-1}(\mu^U_-)=\rho_- \mu^U_-$ for some $\rho_+, \rho_->1$.

Recall that, as shown in \cite{Ka2}, the inclusion $\iota: U\to F_N$
defines a continuous linear map $\iota_\ast: \Curr(U)\to \Curr(F_N)$
which extends the obvious map on conjugacy classes. Namely, for any
nontrivial $u\in U$ we have $\iota_\ast (\eta_u^U)=\eta_u^{F_N}$ where
$\eta_u^U\in \Curr(U)$ and $\eta_u^{F_N}\in \Curr(F_N)$ are the rational
currents defined by $u$ on $U$ and $F_N$ respectively. Moreover, the
map $\iota_*$ has a particularly simple form in a simplicial chart
corresponding to a free basis $A$ of $F_N$ of the form $A=B\sqcup C$
where $B$ is a free basis of $U$. Namely, if we use $A$ as a
simplicial chart on $F_N$ and $B$ as a simplicial chart on $U$ then
for every $\mu\in \Curr(U)$ and every nontrivial freely reduced word
$v\in F(A)$ we have:
\[
\langle v,\iota_*\mu\rangle_A=\begin{cases}\langle v,\mu\rangle_B,
  \quad \text { if } v\in F(B)=U\\
0, \quad\text{ otherwise}.\end{cases}
\]

Put $\nu_+=\iota_\ast(\mu^U_+)$ and $\nu_-=\iota_\ast(\mu^U_-)$.

By the main result of \cite{KL3}, since the supports of $\nu_+$ and
$\nu_-$ are carried by $U$ and since every element of $U$ has
translation length zero on $T_0$, it follows that $\langle T_0, \nu_+
\rangle = \langle T_0, \nu_- \rangle = 0$.

We claim that $\Psi\nu_+=\rho_+\nu_+$ and
$\Psi\nu_-=\rho_-\inv\nu_-$. Indeed, choose a nontrivial element
$a\in U$. We know that (up to rescaling $\nu_+$),
\[
\mu^U_+=\lim_{n\to\infty}
\frac{\Psi|_U^n(\eta_a)}{\rho_+^n}=\lim_{n\to\infty}
\frac{\eta^U_{\Psi^n|_U(a)}}{\rho_+^n}
\]
and therefore by linearity and continuity of $\iota_*$ and using the
fact that $\iota_\ast(\eta_u^U)=\eta_u^{F_N}$ for $u\in U, u\ne 1$, we have:
\[
\nu_+=\iota_\ast(\mu^U_+)=\lim_{n\to\infty} \frac{\eta^{F_N}_{\Psi^n(a)}}{\rho_+^n}.
\]
Using the fact that $\Psi|_U(\mu_+^U)=\rho_+\mu_+^U$, we conclude:
\begin{gather*}
\Psi\nu_+=\lim_{n\to\infty} \Psi
\frac{\eta^{F_N}_{\Psi^n(a)}}{\rho_+^n}=\lim_{n\to\infty}\frac{\eta^{F_N}_{\Psi^{n+1}(a)}}{\rho_+^n}=\\
\lim_{n\to\infty}\frac{\iota_\ast(\eta^{U}_{\Psi|_U^{n+1}(a)})}{\rho_+^n}=\iota_\ast(\lim_{n\to\infty}
\frac{\eta^U_{\Psi|_U^{n+1}(a)}}{\rho_+^n})=\iota_\ast(\Psi|_U(\mu_+^U))=\\
\iota_\ast(\rho_+\mu_+^U)=\rho_+\nu_+.
\end{gather*}
A similar argument shows that $\Psi^{-1}\nu_-=\rho_-\nu_-$.

Thus $T_0$, $\nu_+$ and $\nu_-$ have all the properties required by
the conclusion of the proposition.

\end{proof}

\begin{thm}
\label{iwips-only1}
Let $\phi_0,\psi_0\in \Out(\FN)$ be hyperbolic iwips such that the subgroup
$\langle \phi_0,\psi_0\rangle \subseteq \Out(\FN)$ is not virtually cyclic. Then there
exist $n,m\ge 1$ such that $G=\langle \phi_0^{n},\psi_0^{m}\rangle\subseteq
\Out(\FN)$ is free of rank two and every non-trivial element of $G$ is
a hyperbolic iwip.
\end{thm}

\begin{proof}

We already know by the results of Section~\ref{NS} that if $M_0$ is sufficiently big and $n,m\ge M_0$
 then $G=\langle \phi_0^n,\psi_0^m\rangle$ is free of rank two, that every
nontrivial element of $G$ is atoriodal (i.e. hyperbolic) and that $G$ acts as Schottky group on both, $\CVNbar$ and $\PCurr$ (in the precise sense that the two actions are open 2-generator ping-pong as in Definition \ref{ping-pong}).

Denote $\phi=\phi_0^{M_0}$ and $\psi=\psi_0^{M_0}$.

We define $G_1$ to be the subgroup of $\Out(F_N)$ generated by
$\phi^M$ and $\psi^M$, for $M \geq 1$ large. Every element $\alpha$ of
$G_1 -\{1\}$ is either conjugate in $G$ to some $\phi^k$ or $\psi^k$,
with $k \in \Z \smallsetminus \{0\}$, and thus iwip, or else it is
conjugate in $G$ to a reduced and cyclically reduced word $u$ in
$\phi, \psi$ which contains the subword $\phi^{\pm M}$. After possibly
replacing $\alpha$ by $\alpha^{-1}$, we may assume that $u$ contains
$\phi^M$ as a subword. Hence $u$ is conjugate in $G$ to a reduced word
$u_1$ in $\phi,\psi$ that begins with $\phi^m$ and ends with $\phi^n$
where $m,n\ge \frac{M}{2}-1$ and that represents $\alpha_1\in
\Out(F_N)$ (which is a conjugate of $\alpha$).

Suppose that $\alpha$ is not an iwip, so that $\alpha_1$ is not an iwip either. Then by Proposition~\ref{poles1}
there exist $k\ge 1$, $T_0\in \cvnbar$, $\nu_\pm \in \Curr(F_N)-\{0\}$,
$\rho_\pm >1$ such that $\langle T_0,\nu_+\rangle=\langle T_0,
\nu_-\rangle=0$ and $\alpha_1^k\nu_+=\rho_+\nu_+$,
$\alpha_1^k\nu_-=\rho_-\inv\nu_-$. Then $\alpha_1^k=u_1^k$ and the word
$w=u_1^k$ in $\phi,\psi$ still begins with $\phi^m$ and ends with
$\phi^n$ and has the form $w =u_1^k= \phi^m w' \phi^n$.

Let $[\mu_\pm(\phi)]\in \mathbb P\Curr(F_N)$ be the two fixed
eigencurrents of $\phi$ and let $\mu_\pm(\phi)\in \Curr(F_N)$ be some
representatives of them in $\Curr(F_N)$.
Let $U_+=U_+(\phi)$ and $U_-=U_-(\phi)$ be neighborhoods of
$\mu_+(\phi)$ and $\mu_-(\phi)$ in $\Curr(F_N)$ provided by Lemma
\ref{strictly-bigger}.

If $M$ is big enough, then $n,m\ge \frac{M}{2}-1$ are also big enough so that Proposition \ref{crucial-lemma} can also be applied to $w$ and $w\inv$, and that furthermore, by Lemma \ref{forward-limit-region}, the forward limit region
of $w$ is contained in the image $\Pr U_+$ of the neighborhood $U_+$
of $\mu_+(\phi)$, and similarly for the forward limit region of $w\inv$ and $U_-$.

It follows from Remark \ref{fixed-points} that for every projectively
$w$-fixed current $\mu$ the image $[\mu]$ must be contained in $\Pr
U_+ \cup \Pr U_-$.

The current $\nu_+$ is projectively fixed by $\alpha_1^k=w$ and hence it
is contained in the union of the forward and backward limit regions of
$w$. Therefore $[\nu_+]\in  \Pr U_+\cup \Pr U_-$. Moreover, $\nu_+$ is
projectively fixed by $w$ and is
$w$-expanding (since $\rho_+>1$) while by
Proposition~\ref{crucial-lemma} every projectively fixed current
contained in the backward limit region of $w$ is
$w$-contracting. Hence $[\nu_+]\in \Pr U_+$ so that some non-zero
scalar multiple $\nu_+'$ of $\nu_+$ satisfies $\nu_+'\in U_+$.

A symmetric argument shows that $[\nu_-]\in \Pr U_-$ so that some non-zero
scalar multiple $\nu_-'$ of $\nu_-$ satisfies $\nu_-'\in U_-$.

Thus Lemma \ref{strictly-bigger} applies to $\nu'_+$ and $\nu'_-$,
so that we
obtain
$\langle T_0,\nu'_+\rangle+\langle T_0,\nu'_-\rangle>0$.
But this
contradicts the fact that $\langle T_0,\nu_+\rangle=\langle T_0,
\nu_-\rangle=0$.

Hence $\alpha$ is an iwip, as required, which completes the proof of the
theorem.
${}^{}$
\end{proof}

\begin{cor}\label{cor:sgp}
Let $G\subseteq \Out(F_N)$ be a subgroup which contains some hyperbolic iwip,
and
assume
that $G$ is not virtually cyclic. Then $G$ contains a free
subgroup of rank two
where all
non-trivial elements are hyperbolic iwips.
\end{cor}

\begin{proof}
By Proposition~\ref{prop:dl}, $G$ contains two hyperbolic iwips $\phi$
and $\psi$ some powers of which generate a free group of rank
two. Therefore $\langle \phi,\psi\rangle$ is not virtually cyclic, and
the statement of the corollary follows from Theorem~\ref{iwips-only1}.

\end{proof}


\end{document}